\UseRawInputEncoding 
\documentclass[titlepage,12pt]{article}

\usepackage[utf8]{inputenc}
\usepackage{yfonts}
\usepackage{faktor}
\usepackage{amsfonts}
\usepackage{bbding}
\usepackage{bm}
\usepackage{fancyvrb}

\usepackage{tikz}
\usetikzlibrary{automata, positioning, arrows}

\usepackage[labelsep=space]{caption}

\usepackage{layout}
\usepackage{url}
\usepackage{proof}
\usepackage{mathrsfs}
\usepackage{amsmath,amsthm,amssymb,bussproofs}
\usepackage{braket}
\usepackage{graphicx}
\usepackage{mathtools,stmaryrd}
\usepackage[all]{xy}
\DeclarePairedDelimiter{\ceil}{\lceil}{\rceil}

\makeatletter
\newcommand*\circled[2][1.6]{\tikz[baseline=(char.base)]{
    \node[shape=circle, draw, inner sep=1pt, 
        minimum height={\f@size*#1},] (char) {#2};}}
\makeatother

\newcommand{\CMT}[1]{$\text{CMT}_{#1}$}
\newcommand{\FH}[1]{$\text{FH}_{#1}$}
\renewcommand{\P}{\textbf{P} }
\newcommand{\D}{\textbf{D} }
\newcommand{\PP}{\textbf{P}}
\newcommand{\DD}{\textbf{D}}
\newcommand{\deff}[1]{\textbf{\textit{#1}}}
\newcommand{\goal}[1]{\textbf{\underline{#1}}}

\theoremstyle{definition}
\newtheorem{thm}{Theorem}[section]
\newtheorem{defn}[thm]{Definition}
\newtheorem{cor}[thm]{Corollary}
\newtheorem{prop}[thm]{Proposition}
\newtheorem{question}[thm]{Question}
\newtheorem{lemma}[thm]{Lemma}

\newtheorem{rmk}[thm]{Remark}

\newtheorem{eg}[thm]{Example}

\newcommand{\NN}{\mathbb{N}}

\DeclareMathOperator{\dom}{dom}

\title{Ajtai’s theorem for $T^2_2(R)$ and pebble games with backtracking}
\author{Eitetsu Ken\footnote{
affiliation: Graduate School of Mathematical Sciences, the University of Tokyo\\
supported by: JSPS KAKENHI Grant Number 22KJ1121, Grant-in-Aid for JSPS Fellows, and FoPM program at the University of Tokyo \\
email: yeongcheol-kwon@g.ecc.u-tokyo.ac.jp} \& Mykyta Narusevych\footnote{affiliation: Department of Mathematics and Physics, Charles University\\
email: mykyta.narusevych@matfyz.cuni.cz\\
supported by the Charles University project PRIMUS/21/SCI/014, Charles University Research Centre program No. UNCE/24/SCI/022 and GA UK project No. 246223}}

\begin{document}
\maketitle
\begin{abstract}
    We introduce a pebble game extended by backtracking options for one of the two players (called Prover) and reduce the provability of the pigeonhole principle for a generic predicate $R$ in the bounded arithmetic $T^2_2(R)$ to the existence of a particular kind of winning strategy (called oblivious) for Prover in the game.
    
    While the unprovability of the said principle in $T^2_2(R)$ is an immediate consequence of a celebrated theorem of Ajtai (which deals with a stronger theory $T_2(R)$), up-to-date no methods working for $T^2_2(R)$ directly (in particular without switching lemma) are known.
    
    Although the full analysis of the introduced pebble game is left open, as a first step towards resolving it, we restrict ourselves to a simplified version of the game. In this case, Prover can use only two pebbles and move in an extremely oblivious way. Besides, a series of backtracks can be made only once during a play.
    Under these assumptions, we show that no strategy of Prover can be winning.
    
\end{abstract}

\section{Introduction}

\quad Ajtai's theorem states the unprovability of the pigeonhole principle in bounded arithmetic and a super-polynomial
lower bound for it in constant depth Frege systems (proved in \cite{Ajtai}, and later strengthened by \cite{remake}, \cite{remake2} and \cite{countvsphp}. See Chapter 8 and 15 of \cite{proofcomplexity} for technical details and historical remarks).
The theorem and its variations (\cite{Ajtaiindependence}, \cite{countvsphp}, and \cite{More}) are among the strongest lower-bound results in Proof Complexity and independence results for bounded arithmetics.
One of the main tools to prove them is Switching Lemma for so-called (shallow) $PHP$-trees. (See \cite{Krajicek} or \cite{proofcomplexity} for a clear exposition.)
All known frameworks for understanding the proof of Ajtai's theorem, such as forcing with random variables \cite{forcingwithrandomvariables} and partially definable forcing \cite{partiallydefinableforcing} in addition to original literature, depend heavily on it.

However, we still do not know any appropriate counterpart of Switching Lemma for resolving long-standing open problems such as: $I\Delta_{0}(R)$ (or $AC^{0}$-Frege) v.s. the weak pigeonhole principle or $V^{0}(2)$ (or $AC^{0}(2)$-Frege) v.s. the (injective) pigeonhole principle (see \cite{proofcomplexity} for detailed presentations of the problems).
Even when we compare the injective pigeonhole principle and the full version of onto pigeonhole principle, Switching Lemma seems not that flexible (\cite{ontovsinj}).
Therefore, it would be very interesting if we could provide another proof of Ajtai's theorem, which is Switching-Lemma-free.

In this article, to make the presentation as simple as possible, we focus on $T^{2}_{2}(R)\not\vdash ontoPHP^{n+1}_{n}(R)$, which is an immediate consequence of Ajtai's theorem,  and its propositional translation $LK^{*}_{2+\frac{1}{2}, O(1)}$ v.s. $ontoPHP$.
We present a possible strategy to obtain a Switching-Lemma-free proof of the statement, although it is not fully successful at this point.

Our approach is an extension of pebble-game analysis for resolution (\cite{ResolutionPebble}), which is powerful enough to obtain relatively easy independence result: $T^{1}_{2}(R) \not\vdash ontoPHP^{n+1}_{n}(R)$.
We define a game $\mathcal{G}_{2}$, which is an extended pebble game allowing Spoiler in \cite{ResolutionPebble} (in our terminology, \textbf{Prover}) to \deff{backtrack} the game record bringing a part of the current partial assignment back to the past.
Since \textbf{Prover} wins the game easily if there is no restriction for strategies, we introduce \deff{oblivious strategies} of \textbf{Prover} and show that if $T^{2}_{2}(R) \vdash ontoPHP^{n+1}_{n}(R)$, then some oblivious strategy of \textbf{Prover} beats any strategy of the opponent (in our terminology, \textbf{Delayer}) in the game $\mathcal{G}_{2}$ of appropriate parameters. 

Although we leave the complete analysis of $\mathcal{G}_{2}$ open, at least we analyze a toy case of $\mathcal{G}_{2}$ setting the parameters to easy ones, further restricting oblivious strategies to ``automaton-like'' strategies, and allowing \textbf{Prover} to backtrack just once.
 
We believe one advantage of our approach compared to tackle the functional version of the problem ( adopting fresh function symbols $f$ and $g$ instead of $R$ to describe the bijective pigeonhole principle) via CPLS, the $NP$-search problem corresponding to $T^{2}_{2}$ (\cite{CPLS}), is that it is enough for us to consider partial assignments of small size (of order $\log(n)^{O(1)}$).
 It is at least not immediate for CPLS and 1-Ref(RES) in \cite{CPLS} when we consider dag-like $R(\log)$ refutations which are not necessarily narrow.
Furthermore, our approach differs from those considering $PLS^{NP}$ and GLS problems discussed in \cite{ChiariJan}, the $\Sigma^{b}_{2}$-search problems corresponding to $T^{2}_{2}$; roughly speaking, in our game approach, we consider a concrete way to answer each $NP$-query in $PLS^{NP}$ computations, at the sake of taking exponentially many steps in general (cf. Corollary \ref{determinacy}) but following a reasonable restriction (cf. Definition \ref{obliviousstrategy}). 
On the other hand, the analysis of the game is much more complex than usual pebble games because of the backtracking option.

The article is organized as follows:

\S \ref{Preliminaries} presents the convention and the formal system we shall follow and use.

\S \ref{A review of unprovability of the pigeonhole principle over T12} briefly mentions the base case of Ajtai's theorem to locate our work in the course of previous research explicitly: $T^1_2(R)$ does not prove the pigeonhole principle for relation $R$.

\S \ref{A game notion for T22} introduces our game notion $\mathcal{G}_2$, which is, roughly speaking, a pebbling game allowing \textbf{Prover} to backtrack the records of the game. 
The notion can be regarded as a transposition of Buss's witnessing argument (\cite{Buss}, \cite{clearwitnessing}) to our settings.
We connect formal proofs of the pigeonhole principle in $T^2_2(R)$ and winning strategies of \textbf{Prover} for $\mathcal{G}_2$.

In \S \ref{Analysis of simplified G2}, we consider a very limited case of $\mathcal{G}_2$:
\textbf{Prover} has only two pebbles, has to use extremely oblivious strategies, and is allowed to backtrack only once, while \textbf{Delayer} is of full power. 
We show that, in this setting, \textbf{Delayer} wins $\mathcal{G}_2$.

\S \ref{Acknowledgement} is Acknowledgement.

\S \ref{Appendix} is an appendix.
We sketch the justification of our restriction $p\Sigma_{i}$ and $s\Sigma_{i}$ of the form of propositional formulae adopted in the paper, present another equivalent definition of $\mathcal{G}_{2}$, and discuss another proof strategy on who wins the toy case of the game $\mathcal{G}_{2}$ analyzed in \S \ref{Analysis of simplified G2}.

\section{Preliminaries}\label{Preliminaries}

In this article, we consider a variant of $LK^{*}_{2+\frac{1}{2}}$, a propositional counterpart of the relativized bounded arithmetic $T^{2}_{2}(R)$, whose proofs are of constant heights and consist of cedents of constant cardinalities. 
See Chapter 8 of \cite{proofcomplexity} for the precise formulation and established translations. Note that it is as powerful as quasipolynomial-sized $R(\log)$-proofs when depth-2 formulae are in concern. See Chapter 10 of \cite{proofcomplexity}.

To clarify the convention, we explicitly describe our setup.

For each $i \in \NN$, we adopt the following common notations: $|i| := \ceil{\log_{2}(i+1)}$, the length of the binary representation of $i$, and $[i] := \{1, \ldots, i\}$ ($[0] := \emptyset$).

For a family $\mathcal{F}$ of sets and $s \in \NN$, $\mathcal{F}_{\leq s}$ denotes the subfamily of $\mathcal{F}$ collecting the elements of cardinality $\leq s$.

We adopt $\lnot$, binary $\lor,\land$ and unbounded $\bigvee, \bigwedge$ as propositional connectives, $0$ and $1$ as propositional constants, and consider only propositional formulae of negation normal form.
Given a formula $\varphi$, $\overline{\varphi}$ denotes the canonical negation normal form of $\lnot \varphi$. It is called \deff{the complement of $\varphi$}.

For a propositional formula $\varphi$, $|\varphi|$ denotes the size of $\varphi$, say, the number of occurrences of variables and connectives in $\varphi$. 
The precise definition does not matter as long as the conventions are polynomially related. See also Chapter 1 of \cite{proofcomplexity}.

The distinction between unbounded $\bigvee, \bigwedge$ and binary $\lor,\land$ are adopted since the following variation of $\Sigma^{S,t}_{d}$ given at the end of section 3.4 of \cite{proofcomplexity} is in our mind:
a propositional formula $\varphi$ is $p\Sigma_{i+\frac{1}{2}}(z)$ ($i,z \in \NN$) if and only if it has the following form:
\[
\varphi = \underbrace{\bigvee_{j_{1} \in J_{1}}\bigwedge_{j_{2} \in J_{2}}\cdots}_{\mbox{exactly} \ i\  \mbox{-times}}\psi_{j_{1},\ldots, j_{i}},
\]
 where
 \[ |\varphi| \leq z \ \& \ \ \forall k \in [i].\ J_{k} \neq \emptyset \]
 and each $\psi_{\vec{j}}$ is $\bigvee$- and $\bigwedge$-free and satisfies $|{\psi_{\vec{j}}}| \leq \log z$.
 Note that if $\varphi$ is $p\Sigma_{i+\frac{1}{2}}(z)$, then such $i$ is unique since we distinguish $\bigvee, \bigwedge$ from $\lor, \land$.
 $p\Pi_{i+\frac{1}{2}}(z)$ is defined similarly, switching the roles of $\bigvee$ and $\bigwedge$.
 
 We say $\varphi$ is $s\Sigma_{i+\frac{1}{2}}(z)$ if it is $p\Sigma_{i'+\frac{1}{2}}(z)$ for some $0 \leq i^{\prime} \leq i$, or it is $p\Pi_{i'+\frac{1}{2}}(z)$ for some $0 \leq i^{\prime} < i$.
 Similarly for $s\Pi_{i+\frac{1}{2}}(z)$. 
 We are particularly interested in the case $z=2^{|n|^{O(1)}}$ for parameters $n$ and $i \leq 2$.

We clarify our treatment of \textit{trees} in our descriptions of sequent calculus and of our game.
\begin{defn}
Let $b,h \in \NN$.
$[b]^{\leq h}$ denotes the set of finite sequences on $[b]$ with length $\leq h$ (including the empty sequence).
 
 For $v, w \in [b]^{\leq h}$, $v \subseteq w$ means that $v$ is a initial segment of $w$ (or $w$ is an extension of $v$).
 
 The length of $v$ is denoted by $height(v)$.
 
 Given $v \in [b]^{\leq h}$ and $w \in [b]^{\leq k}$, $v*w$ denotes the concatenation lying in $[b]^{\leq h+k}$. 
 
 For $v=(a_{1},\ldots, a_{h}) \in [b]^{h}$ and $k \in \NN$, set
 \begin{align*}
  v_{k} := \begin{cases}
   a_{k} \quad &(k \leq h) \\
   -1 \quad &(k > h)
  \end{cases}.
 \end{align*}
 
 If $k \leq h$, then we define $v_{\leq k}:=(v_{1},\ldots, v_{k})$.
\end{defn}
\begin{rmk}
 We often identify $[b]$ as $[b]^{1}$ and abuse the notation.
 For example, we write $v*i$ for $v*(i)$.   
\end{rmk}
\begin{defn}
\deff{A rooted tree of height $\leq h$ and $\leq b$-branching} is a subset $T \subseteq [b]^{\leq h}$ such that:
\begin{enumerate}
\item $T \neq \emptyset$.
 \item $v \subseteq w \in T \Longrightarrow v \in T$.
\end{enumerate}
$\emptyset \in T$ is called \deff{the root of $T$}.

For $v, w \in T$, $v < w$ means \deff{the lexicographic order}:
\begin{align*}
v < w : \Longleftrightarrow \exists k \leq h.\ (v_{\leq k} = w_{\leq k} \ \& \ v_{k+1} < w_{k+1})
\end{align*}

If $v*i \in T$, then we say \deff{$v*i$ is a child of $v$}, and \deff{$v$ is the parent of $v*i$}. If $v*i, v*j \in T$, then they are said to be \deff{siblings}.

When $v, w \in T$ satisfy $height(v) = height(w)$, then we say \deff{$v$ is left of $w$} when $v < w$.

Put 
\[height(T) := \max_{v \in T} height(v).\]
\end{defn}

\begin{eg}
If $v*1 \in T$, then $v*1$ must be the leftmost child of $v$ in $T$.
\end{eg}

In this article, following the convention of \cite{OrdinalAnalysis}, we use a one-sided formulation of propositional sequent-calculus, which can be regarded as the non-uniform version of a first-order sequent calculus implementing $T^{2}_{2}(R)$:

\begin{defn}\label{cedent calculus}
\deff{A cedent} is a finite set of propositional formulae. 
Given a cedent $\Gamma$, its \deff{semantic interpretation} is the propositional formula $\bigvee_{\varphi \in \Gamma}\varphi$. 
Here, if $\Gamma=\emptyset$, we set $\bigvee_{\varphi \in \Gamma}\varphi:=0$.
Under a truth assignment, $\Gamma$ is said to be \deff{true} if and only if its semantic interpretation is true.

We often denote cedents of the form 
\[\Gamma_{1} \cup \ldots \cup \Gamma_{k}\cup \{\varphi_{1},\ldots, \varphi_{m}\}\]
 by 
 \[\Gamma_{1}, \ldots, \Gamma_{k}, \varphi_{1},\ldots, \varphi_{m}.\]
\end{defn}

Given constants $c,d$, a cedent $S$ and a finite vertex-labeled tree $\pi=(\mathcal{T},\mathcal{S})$, $\pi$ is \deff{a $LK^{*}_{d+\frac{1}{2},c}$-derivation of $S$ (without redundancy)} if and only if the following hold:
\begin{enumerate}
 \item $height(\mathcal{T}) \leq c$.
 \item For each $v \in \mathcal{T}$, $\mathcal{S}(v)$ is a cedent of cardinality $\leq c$.
 \item $\mathcal{S}(\emptyset) = S$.
 \item For each $v \in \mathcal{T}$, $\mathcal{S}(v)$ is derived from the labels of its children, that is, $(\mathcal{S}(v*i))_{v*i \in \mathcal{T}}$ by applying one of the following derivation rules:
 \begin{itemize}
 \item Initial cedent:
 
 \begin{prooftree}
 \AxiomC{}  \RightLabel{\quad (where $x$ is a constant or a variable)}
  \UnaryInfC{$\Gamma, x, \overline{x}$}

\end{prooftree}
 
 \item $\bigvee$-Rule: 
   \begin{prooftree}
 \AxiomC{$\Gamma, \varphi_{i_{0}}$} \RightLabel{\quad (where $\bigvee_{i=1}^{I}\varphi_{i} \in \Gamma$, $1\leq i_{0} \leq I$, $\varphi_{i_{0}} \not \in \Gamma$)}
  \UnaryInfC{$\Gamma$}
 \end{prooftree}
 
\item $\lor$-Rule: 
   \begin{prooftree}
 \AxiomC{$\Gamma, \varphi_{i_{0}}$} \RightLabel{\quad (where $\varphi_{1} \lor \varphi_{2} \in \Gamma$, $i_{0}=1$ or $i_{0}=2$, $\varphi_{i_{0}} \not \in \Gamma$)}
  \UnaryInfC{$\Gamma$}
 \end{prooftree}

 \item $\bigwedge$-Rule:
    \begin{prooftree}
 \AxiomC{$\Gamma, \varphi_{1}$}
 \AxiomC{$\Gamma, \varphi_{2}$}
 \AxiomC{$\cdots$}
 \AxiomC{$\Gamma, \varphi_{I}$}
   
  \QuaternaryInfC{$\Gamma$}
 \end{prooftree}
 where $\bigwedge_{i=1}^{I}\varphi_{i} \in \Gamma$, and $\varphi_{i} \not \in \Gamma$ for each $i \in [I]$.
 
 \item $\land$-Rule:
    \begin{prooftree}
 \AxiomC{$\Gamma, \varphi_{1}$}
 \AxiomC{$\Gamma, \varphi_{2}$}
  \BinaryInfC{$\Gamma$}
 \end{prooftree}
 where $\varphi_{1} \land \varphi_{2} \in \Gamma$, and $\varphi_{i} \not \in \Gamma$ for each $i=1,2$.

 \item Trivial Cut:
  \begin{prooftree}
 \AxiomC{$\Gamma, 0$} \RightLabel{\quad (where $0 \not \in \Gamma$)}
  \UnaryInfC{$\Gamma$}
 \end{prooftree}

 \item $p\Sigma_{d+\frac{1}{2}}$-Induction:
     \begin{prooftree}
 \AxiomC{$\Gamma, \varphi_{1}$}
 \AxiomC{$\Gamma, \overline{\varphi_{1}},\varphi_{2}$}
 \AxiomC{$\cdots$}
 \AxiomC{$\Gamma, \overline{\varphi_{I-1}},\varphi_{I}$}
  \AxiomC{$\Gamma, \overline{\varphi_{I}}$}
  \QuinaryInfC{$\Gamma$}
 \end{prooftree}
 where each $\varphi_{i}$ is $p\Sigma_{d+\frac{1}{2}}(|{\pi}|)$, and $\varphi_{i}, \overline{\varphi}_{i}  \not \in \Gamma$ for $i \in [I]$.
 Here, we define the size $|{\pi}|$ of the proof $\pi$ as 
 \[|{\pi}| := \sum_{v \in \mathcal{T}} \sum_{\varphi \in \mathcal{S}(v)} |{\varphi}|.\]
Note that, when $I=1$, $p\Sigma_{d+\frac{1}{2}}$-Induction is a usual cut-rule for $p\Sigma_{d+\frac{1}{2}}(|{\pi}|)$-formulae.
\end{itemize}
\end{enumerate}
\begin{rmk}
In this article, we are particularly interested in $d=1,2$.
\end{rmk}


\begin{defn}
$ontoPHP^{n+1}_{n}$ is the following cedent\footnote{We may have expanded $\bigvee$ to lower the complexity of the formulae in the cedent, but we are interested in cedents of constant cardinality with respect to the parameter $n$ (recall Definition \ref{cedent calculus} and see Theorem \ref{ProofIsStrategy2}), so we formulated the cedent $ontoPHP^{n+1}_{n}$ as above.}:
\begin{align*}
\left\{ \bigvee_{p \in [n+1]}\bigwedge_{h \in [n]} \lnot r_{ph}, \bigvee_{\substack{p \neq p^{\prime} \in [n+1],\\ h \in [n]}} (r_{ph} \land r_{p^{\prime}h}),
 \bigvee_{h \in [n]}\bigwedge_{p \in [n+1]} \lnot r_{ph}, \bigvee_{\substack{h \neq h^{\prime} \in [n],\\ p \in [n+1]}} (r_{ph} \land r_{ph^{\prime}}) \right\}
\end{align*}
\end{defn}

We can also consider a first-order counterpart of the above calculus, and we can regard the above calculus as a propositional translation of it.
For example,

\begin{prop}[essentially, Corollary 9.1.4 and the proof of Lemma 9.5.1 in \cite{Krajicek}]\label{ParisWilkieTranslation}
Suppose $T^{d}_{2}(R) \vdash ontoPHP^{n+1}_{n}(R)$, where $ontoPHP^{n+1}_{n}(R)$ is a natural formalization of the pigeonhole principle for bijections using $(n+1)$ pigeons and $n$ holes.
 Then $ontoPHP^{n+1}_{n}$ has $2^{|n|^{O(1)}}$-sized $LK^{*}_{d+\frac{1}{2}, O(1)}$-proofs.
\end{prop}

We omit the proof since it is just an instance of another formulation of classical Paris-Wilkie translation \cite{Delta0PHP}, applied to the first-order proof systems corresponding to $T^d_2(R)$ presented in Appendix \ref{firstorderLK} and the propositional proof system $LK^*_{d+\frac{1}{2},O(1)}$.
See \cite{Krajicek} for a comprehensive presentation of Paris-Wilkie translation and \cite{clearwitnessing} for a presentation of the closest conventions to ours.

\begin{rmk}
 Note that our convention $LK^{*}_{d+\frac{1}{2}, O(1)}$ adopts $p\Sigma_d$-Induction rule, 
 which has unbounded arity and makes the translation of $\Sigma^b_d(R)$-Induction more straightforward than the argument in the proof of Theorem 9.1.3 in \cite{Krajicek}.
 The resulting proof tree is of constant height, and each cedent is of constant cardinality.
\end{rmk}

\begin{rmk}\label{subformulaproperty}
$2^{|n|^{C}}$-sized $LK^{*}_{d+\frac{1}{2}, O(1)}$-proofs of $ontoPHP^{n+1}_{n}$ consist only of $s\Sigma_{d+\frac{1}{2}}(2^{|n|^{C}})$-formulae, $s\Pi_{d+\frac{1}{2}}(2^{|n|^{C}})$-formulae, and subformulae of the formulae in $ontoPHP^{n+1}_{n}$.
\end{rmk}

\section{A review of unprovability of the pigeonhole principle over $T^{1}_{2}(R)$}\label{A review of unprovability of the pigeonhole principle over T12}
In this section, we briefly review the independence result, which can be regarded as the base case of Ajtai's Theorem, in order to locate this article in the course of previous research:
\begin{thm}[essentially by \cite{Delta0PHP}. See for example \cite{partiallydefinableforcing} or \cite{Mykytaforcing} for clear expositions]\label{BaseIndependence}
\[T^{1}_{2}(R) \not \vdash ontoPHP^{n+1}_{n}(R).\]
\end{thm}

It can be proven in multiple ways: by Paris-Wilkie forcing (\cite{Delta0PHP}. See \cite{partiallydefinableforcing} for a neat presentation), by Riis criterion (\cite{RiisCriterion}), and by converting a $T^{1}_{2}(R)$-proof of $ontoPHP^{n+1}_{n}(R)$ to narrow resolution proofs (\cite{reslogtonarrowresol}) and applying resolution width lower bounds for $ontoPHP^{n+1}_{n}$ obtained by pebble-game analysis (\cite{ResolutionPebble}). 

Our approach can be regarded as a generalization of the last proof strategy to $T^{2}_{2}(R)$.
We present our version of the proof of Theorem \ref{BaseIndependence} in the rest of this section for completeness and as a showcase of the more complex argument employed in the proof of Theorem \ref{ProofIsStrategy2}.

\begin{defn}
 
Let $\mathcal{M}_{n}$ be \deff{the set of all partial matchings between $[n+1]$ (pigeons) and $[n]$ (holes)}.

For $M, M' \in \mathcal{M}_{n}$, when $M\cup M' \not\in \mathcal{M}_{n}$, in other words, $M$ and $M'$ match a pigeon to different holes or a hole to different pigeons, we say \deff{$M$ contradicts $M'$}, denoted by $M \perp M'$.
We say \deff{$M$ is consistent with $M'$} otherwise.
\end{defn}

We define the following variation of a pebble game.
\begin{defn}\label{DefG1}
Given $n, C \in \NN$, $\mathcal{G}_{1}(n,C)$ is the following game:

  \begin{enumerate}
  \item Played by two players. We call them \textbf{Prover} and \textbf{Delayer}.
  For readability, we use the pronoun ``\textit{he}'' for \textbf{Prover} and ``\textit{she}'' for \textbf{Delayer}.
  \item Possible positions are sequences on $(\mathcal{M}_{n})_{\leq |n|^{C}}$ with length at most $2^{|n|^{C}}$.
  \item The initial position is the sequence $(\emptyset)$ of length $1$.
  \item Now, we describe transitions between positions together with each player's options.
   suppose the current position is $(M_{0}, \ldots, M_{l}) \in ((\mathcal{M}_{n})_{\leq |n|^{C}})^{l+1}$.
    If $l+1 \geq 2^{|n|^{C}}$, then the play ends and \textbf{Delayer} wins.
    If $l+1 < 2^{|n|^{C}}$, then proceed as follows:
   \begin{enumerate}
     \item\label{query} First, \textbf{Prover} plays a subset $Q \subseteq P_{n} \dot\cup H_{n}$ of cardinality $\# Q \leq |n|^{C}$, and sends it to \textbf{Delayer}.   
   \item\label{answer} \textbf{Delayer} plays a matching $M' \in (\mathcal{M}_{n})_{\leq |n|^{C}}$ such that $M'$ is a minimal matching covering $Q$ and is consistent with $M_{l}$.
   If such $M'$ does not exist, the play ends and \textbf{Prover} wins.
  Otherwise, the sequence $(M_{0}, \ldots, M_{l}, M')$ is the next position.
\end{enumerate}
  \end{enumerate}
\end{defn}
The game is determined since the length of a position increases by $1$ as long as the play continues and the game ends when the length reaches $2^{|n|^{C}}$.

\begin{rmk}
An informal description of $\mathcal{G}_{1}(n,C)$ is as follows: \textbf{Prover} has $2|n|^{C}$ pebbles, and there is a game board having $(n+1)$ pigeons and $n$ holes.
He puts some pebbles on some of the pigeons and the holes in a course of the game.
Initially, there is no pebbles on the board.
In each turn, \textbf{Prover} selects at most $|n|^{C}$ pigeons and holes and puts one pebble per each.
Then \textbf{Delayer} answers which holes (resp. pigeons) are matched to the pebbled pigeons (resp. holes).
She must pretend that she knows a bijection between the pigeons and holes, that is, the answers to the pebbles on the board should always be a partial matching between pigeons and holes.
If she cannot answer in this way, she loses.
If there are more than $|n|^{C}$ pebbles on the board, \textbf{Prover} must remove some so that the number eventually becomes no more than $|n|^{C}$.
Note that we only refer the answers to the pebbles on the current board, and the answers to the removed pebbles become irrelevant to the winning condition anymore.
\end{rmk}

Moreover, the following is an easy observation:
\begin{lemma}
 If $2|n|^{C} \leq n$, then \textbf{Delayer} has a winning strategy for $\mathcal{G}_{1}(n,C)$.
 Here, \deff{a winning strategy of} \textbf{Delayer} \deff{for $\mathcal{G}_{1}(n,C)$} is a mapping \[f \colon \bigcup_{l=0}^{2^{|{n}|^{C}}-1}\left((\mathcal{M}_{n})_{\leq |n|^{C}}\right)^{l} \times (2^{P_{n} \dot\cup H_{n}})_{\leq |n|^{C}} \rightarrow (\mathcal{M}_{n})_{\leq |n|^{C}}\]
  such that \textbf{Delayer} always wins $\mathcal{G}_{1}(n,C)$ if she plays $f(P,Q)$ when the current position is $P$ and \textbf{Prover} queries $Q$, regardless \textbf{Prover}'s moves.
\end{lemma}
\begin{proof}
In each turn, given a position $P=(M_{0}, \ldots, M_{l})$ and a query $Q$, there exists a matching $\tilde{M}$ such that $M_{l} \subseteq \tilde{M}$, $\tilde{M}$ covers $Q$, and $\# \tilde{M} \leq 2 |n|^{C}$ since  $2|n|^{C} \leq n$.
Therefore, it suffices for \textbf{Delayer} to answer minimal $M' \subseteq \tilde{M}$ which covers $Q$.
In this way, \textbf{Delayer} can survive arbitrary many turns and win the game.
\end{proof}

\begin{cor}
Given $C \in \NN$, \textbf{Delayer} has a winning strategy for $\mathcal{G}_{1}(n,C)$ for sufficiently large $n$.
\end{cor}

On the other hand, we have the following:
\begin{thm}[essentially by \cite{reslogtonarrowresol}]\label{ProofIsStrategy1}
Suppose $ontoPHP^{n+1}_{n}$ has $LK^{*}_{1+\frac{1}{2},O(1)}$-proofs of size $2^{|n|^{O(1)}}$.
Then there exists a constant $C > 0$ such that, for sufficiently large $n \in \NN$, \textbf{Prover} has a winning strategy for $\mathcal{G}_{1}(n,C)$.
Here, \deff{a winning strategy of} \textbf{Prover} \deff{for $\mathcal{G}_{1}(n,C)$} is a mapping $f \colon \bigcup_{l=0}^{2^{|{n}|^{C}}-1}\left((\mathcal{M}_{n})_{\leq |n|^{C}}\right)^l \rightarrow (2^{P_{n} \dot\cup H_{n}})_{\leq |n|^{C}}$ such that \textbf{Prover} always wins $\mathcal{G}_{1}(n,C)$ if he queries $f(P)$ when the current position is $P$, regardless \textbf{Delayer}'s moves.
\end{thm}

The proof can be regarded as a ``restriction'' of the proof of Theorem \ref{ProofIsStrategy2}, and we only give a sketch here.
In the rest of this section, we write ``$\Sigma_{d}$'' (resp. ``$\Pi_{d}$'') instead of ``$\Sigma_{d+\frac{1}{2}}(|\pi|)$'' (resp. ``$\Pi_{d+\frac{1}{2}}(|\pi|)$'') for readability.

\begin{defn}\label{defofforcing}
 Let $M \in \mathcal{M}_{n}$.
 Let $\varphi$ be a $s\Sigma_{0}$-formula.
 We write $M \Vdash \varphi$ to denote that the partial assignment $\rho^{M}$ induced by $M$ covers the all variables occuring in $\varphi$ and $\rho^{M} \models \varphi$.
 Here, $\rho^{M}$ is the following assignment:
 \begin{align*}
 \rho^{M}\colon r_{ph} \mapsto
 \begin{cases}
  1 \quad &(\mbox{if $M$ matches $p$ and $h$}) \\
  0 \quad &(\mbox{if $M \perp \{p \mapsto h\}$})\\
  \mbox{undefined} \quad &(\mbox{otherwise})
 \end{cases}
 \end{align*}
\end{defn}

\begin{defn}\label{forcingforcedents}
Let $\Gamma$ be a cedent consisting only of $s\Sigma_{2}$- or $s\Pi_{2}(z)$-formulae.
Let $\Pi(\Gamma)$ be the set of all $p\Pi_{1}(z)$- or $p\Pi_{2}(z)$-formulae in $\Gamma$.
Let $W$ be a function defined on $\Pi(\Gamma)$ and $M \in \mathcal{M}_{n}$.
We say \deff{$(M,W)$ falsifies $\Gamma$} when the following hold:
\begin{itemize}
\item For each $\varphi \in \Gamma$ which are $p\Sigma_{0}(z)$, $M \Vdash \overline{\varphi}$.

\item For each $\bigwedge_{i=1}^{I} \varphi_{i} \in \Pi(\Gamma)$, $w:=W(\bigwedge_{i=1}^{I} \varphi_{i}) \in [I]$.
Moreover, if $\bigwedge_{i=1}^{I} \varphi_{i}$ is $p\Pi_{1}(z)$ (and therefore $\varphi_{i}$ is $p\Sigma_{0}$), then $M \Vdash \overline{\varphi_{w}}$.
\end{itemize}
We call $W$ \deff{a counterexample function for $\Gamma$}.
\end{defn}


\begin{proof}[Proof of Theorem \ref{ProofIsStrategy1}]
 By assumption, there exists $C>0$ such that, for any sufficiently large $n$, there exists an $LK^{*}_{1+\frac{1}{2},C}$-derivation $\pi=(\mathcal{T},\mathcal{S})$ of $ontoPHP^{n+1}_{n}$ satisfying $|\pi| \leq 2^{|n|^{C}}$.
 We extract a winning strategy of \textbf{Prover} from $\pi$.
 The main idea is the same as the PLS witnessing for \textbf{Prover} given in \cite{clearwitnessing}; in the course of a play, \textbf{Prover} always looks at a vertex $v[P]$ in $\mathcal{T}$ and tries to falsify $\mathcal{S}(v[P])$ by $(M_{l},W[P])$, where we suppose $P=(M_{0},\ldots,M_{l})$ is the current position, and $W[P]$ is the counterexample function for $\mathcal{S}(v[P])$ obtained in the play so far.

 Precisely speaking, we inductively construct \textbf{Prover}'s strategy 
 \[f \colon \bigcup_{l=0}^{2^{|{n}|^{C}}-1}\left((\mathcal{M}_{n})_{\leq |n|^{C}}\right)^l \rightarrow (2^{P_{n} \dot\cup H_{n}})_{\leq |n|^{C}}\]
 together with a vertex function \[v \colon \bigcup_{l=0}^{2^{|{n}|^{C}}-1}\left((\mathcal{M}_{n})_{\leq |n|^{C}}\right)^l \rightarrow \mathcal{T};\ P \mapsto v[P]\]
 and a counterexample functional
 \[W \colon \bigcup_{l=0}^{2^{|{n}|^{C}}-1}\left((\mathcal{M}_{n})_{\leq |n|^{C}}\right)^l \rightarrow (\Pi_1(\mathcal{S}(v[P]) \rightarrow \NN);\ P \mapsto W[P].\]
(Here, for a cedent $S$, $\Pi_1(S)$ denotes the set of all the $p\Pi_1$-formulae in $S$.
Note that $\dom(W[P])$ depends on the input $P$ and the definition of $v[P]$.)
They will be designed so that they satisfy the following condition (\ref{IH}) for all the possible positions $P$ following the strategy $f$:
\begin{defn}\label{IH}
Let $P=(M_{0},\ldots, M_{l})$ be a position,
\begin{itemize}
 \item all the $s\Pi_1$-formulae in $\mathcal{S}(v[P])$ are falsified by $(M_l, W[P])$.
 \item for every $v' \subsetneq v[P]$, there exists $i<l$ such that $v[P']=v'$, where $P':=(M_0,\ldots,M_i)$.
\end{itemize}
\end{defn}

 For the initial position $P_0:=(\emptyset)$, we set $v[P_0]:=\emptyset \in \mathcal{T}$.
 The cedent $\mathcal{S}(v[P_0])$ is $ontoPHP^{n+1}_{n}$, and therefore all the $s\Pi_{1}$-formulae are falsified by $M_{0}=\emptyset$ trivially since there is none.
 Hence, we set $W[P_0]$ as the empty map.

Now, we proceed by induction on the length of positions.
 Assume $P=(M_{0},\ldots, M_{l})$ is a possible position inside the current domains of the partial functions $v$ and $W$ constructed so far but outside the current domain of the partial function $f$ already constructed. 
 Furthermore, assume that $f,v,W$ satisfy the condition (\ref{IH}) for all the positions in the current domains.
 Then we set $f(P)$ and $v[P*M_{l+1}], W[P*M_{l+1}]$ (where $P*M_{l+1}$ is the next positions of $P$ determined by the query $f(P)$ and \textbf{Delayer}'s answer $M_{l+1}$ which is consistent with $M_l$) depending on the inference rule used to derive $\mathcal{S}(v[P])$ (Note that the induction hypothesis is maintained in each case):
 
  \begin{itemize}
 \item Note that $\mathcal{S}(v[P])$ cannot be an Initial cedent:
 \begin{prooftree}
 \AxiomC{}  \RightLabel{\quad (where $x$ is a constant or a variable)}
  \UnaryInfC{$\Gamma, x, \overline{x}$}
\end{prooftree}
since $x$ and $\overline{x}$ cannot be falsified simultaneously.

 \item When $\mathcal{S}(v[P])=\Gamma$ is derived by $\lor$-Rule: 
   \begin{prooftree}
 \AxiomC{$\Gamma, \varphi_{i_{0}}$} \RightLabel{\quad (where $\varphi_{1} \lor \varphi_{2} \in \Gamma$, $i_{0}=1$ or $i_{0}=2$, $\varphi_{i_{0}} \not \in \Gamma$)}
  \UnaryInfC{$\Gamma$}
 \end{prooftree}
 Since $\varphi_{1} \lor \varphi_{2}$ is $s\Sigma_{2}$ or $p\Pi_{1}$ by Remark \ref{subformulaproperty}, and its outermost connective is $\lor$ (not $\bigvee$), it must be $s\Sigma_{0}$.
Therefore, it is already falsified by $M_{l}$. 
 Thus it suffices to set $f(P):=\emptyset$, define $v[P*M_{l+1}]$ to be the child of $v[P]$, and set $M_{l+1}:=M_{l}$,  $W[P*M_{l+1}]:=W[P]$.
 
  \item When $\mathcal{S}(v[P])=\Gamma$ is derived by Trivial Cut or $\land$-rule:
 similar to the above.

 \item When $\mathcal{S}(v[P])=\Gamma$ is derived by $\bigvee$-Rule: 
   \begin{prooftree}
 \AxiomC{$\Gamma, \varphi_{i_{0}}$} \RightLabel{\quad (where $\bigvee_{i=1}^{I}\varphi_{i} \in \Gamma$, $1\leq i_{0} \leq I$, $\varphi_{i_{0}} \not \in \Gamma$)}
  \UnaryInfC{$\Gamma$}
 \end{prooftree}
 
  If $\bigvee_{i=1}^{I}\varphi_{i}$ is one of the formulae in $ontoPHP^{n+1}_{n}$, which is the main difference from the situation of PLS witnessing in \cite{clearwitnessing}, $\varphi_{i_{0}}$ corresponds to a pigeon or a hole, and \textbf{Prover} queries it. 
 The answer must falsify $\varphi_{i_{0}}$, and \textbf{Prover} climbs up to the child of $v[P]$, storing the answer as a counterexample if necessary; we treat the cases for the first two formulae in $ontoPHP^{n+1}_{n}$ here:
 \begin{itemize}
  \item If $\bigvee_{i=1}^{I}\varphi_{i}=\bigvee_{p \in [n+1]}\bigwedge_{h \in [n]} \lnot r_{ph}$, then we set $f(P):=\{i_0\}\cup Q$, where $Q$ is the set of pigeons appearing in a minimal matching $M \subseteq M_l$ such that $(M,W[P])$ falsifies all the $s\Pi_1$-formulae in $\mathcal{S}(v[P])$.

  Let $M_{l+1}$ be \textbf{Delayer}'s answer.
  It includes a matching of the form $i_{0} \mapsto h_{0}$.
  Then we set $v[P*M_{l+1}]$ as the child of $v[P]$.
  Furthermore, we set
  $W[P*M_{l+1}]:=W[P] \dot\cup \{\varphi_{i_{0}} \mapsto h_{0}\}$.
   
  \item If $\bigvee_{i=1}^{I}\varphi_{i}=\bigvee_{\substack{p \neq p^{\prime} \in [n+1],\\ h \in [n]}} (r_{ph} \land r_{p^{\prime}h})$, $i_{0}$ is in the form of $\langle p,p',h \rangle$.
  \textbf{Prover} queries $f(P):=\{p,p'\} \cup Q$.
  Let $M_{l+1}$ be \textbf{Delayer}'s answer. 
  It includes matchings of the form $p \mapsto h_{0}$ and $p' \mapsto h_{1}$.
  Then we set $v[P*M_{l+1}]$ as the child of $v[P]$, and $W[P*M_{l+1}]:=W[P]$.
 \end{itemize}

 Otherwise, by Remark \ref{subformulaproperty}, $\bigvee_{i=1}^{I}\varphi_{i}$ is $p\Sigma_{1}$, then $\varphi_{i_{0}}$ is $p\Sigma_{0}$, and therefore there exists at most $|n|^{C}$-many pigeons and holes, say, a set $Q$, such that any $M \in \mathcal{M}_{n}$ covering $Q$ satisfies $M \Vdash \varphi_{i_{0}}$ or $M \Vdash \overline{\varphi_{i_{0}}}$.
 \textbf{Prover} queries $f(P):=Q$, and let $M_{l+1}$ be \textbf{Delayer}'s answer.
 If $M \Vdash \overline{\varphi_{i_{0}}}$, then we set $v[P*M_{l+1}]$ as the child of $v[P]$ and $W[P*M_{l}]:=W[P]$.
 If $M \Vdash \varphi_{i_{0}}$, then there exists the ancestor $v'$ of $v[P]$ where $\bigvee_{i=1}^{I}\varphi_{i}$ is eliminated by $p\Sigma_{1}$-Induction.
 More formally, by induction hypothesis, there exists the position $P'=(M_{0}, \ldots, M_{l'})$ ($l'<l$) such that $v'=v[P']$.
Using this, we set $v[P*M_{l+1}]$ as the immediate right sibling of $v'$ and
 \[W[P*M_{l+1}]:=W[P'] \dot\cup \left\{\bigwedge_{i=0}^{I}\overline{\varphi_{i}} \mapsto i_{0}\right\}\]
 
 \item When $\mathcal{S}(v[P])=\Gamma$ is derived by $\bigwedge$-Rule:
    \begin{prooftree}
 \AxiomC{$\Gamma, \varphi_{1}$}
 \AxiomC{$\Gamma, \varphi_{2}$}
 \AxiomC{$\cdots$}
 \AxiomC{$\Gamma, \varphi_{I}$}
  \QuaternaryInfC{$\Gamma$}
 \end{prooftree}
 where $\bigwedge_{i=1}^{I}\varphi_{i} \in \Gamma$, and $\varphi_{i} \not \in \Gamma$ for each $i \in [I]$.
 We have assumed that the $s\Pi_{1}$-formulae in $\mathcal{S}(v[P])$ are already falsified, so we set $f(P):=\emptyset$, \[v[P*M_{l+1}]:=v[P]*W[P]\left(\bigwedge_{i=1}^{I}\varphi_{i}\right),\] and $W[P*M_{l+1}]:=W[P]$.
  
 \item When $\mathcal{S}(v[P])=\Gamma$ is derived by $p\Sigma_{1}$-Induction:
     \begin{prooftree}
 \AxiomC{$\Gamma, \varphi_{1}$}
 \AxiomC{$\Gamma, \overline{\varphi_{1}},\varphi_{2}$}
 \AxiomC{$\cdots$}
 \AxiomC{$\Gamma, \overline{\varphi_{I-1}},\varphi_{I}$}
  \AxiomC{$\Gamma, \overline{\varphi_{I}}$}
  \QuinaryInfC{$\Gamma$}
 \end{prooftree}
 where each $\varphi_{i}$ is $p\Sigma_{1}$, and $\varphi_{i}, \overline{\varphi}_{i}  \not \in \Gamma$ for $i \in [I]$.
 
 We set $f(P):=\emptyset$.
 $v[P*M_{l+1}]$ is set as the leftmost child with $\Gamma \cup \{\varphi_{1}\}$, which does not have a new $p\Pi_{1}$ formula.
 Thus we set $W[P*M_{l+1}]:=W[P]$.
\end{itemize}

We can observe that, following the strategy $f$, $v[P] <_{lex} v[P*M_{l+1}]$ on $\mathcal{T}$.
Note that the definitions of $f,v,W$ for positions $P$ not covered by the above construction do not matter since they are not reachable if \textbf{Prover} follows the strategy $f$.
Therefore, the play ends with less than $2^{|n|^{C}}$ turns, and
it implies \textbf{Prover} wins the game.

 
\end{proof}

\section{A game notion for $T^{2}_{2}(R)$}\label{A game notion for T22}

\quad This section aims to give a candidate of an appropriate game notion for $2^{|n|^{O(1)}}$-sized $LK^{*}_{2+\frac{1}{2},O(1)}$-proofs.

Towards precise descriptions of them, we define the following notions:

\begin{defn}
A tree $T$ is \deff{an $(n,C)$-tree} if and only if it satisfies the following:
\begin{enumerate}
 \item $T$ is a tree of height $\leq C$ and $\leq 2^{|n|^{C}}$-branching,
 \item If $v*k \in T$ and $l \in [k]$, then $v*l \in T$.
\end{enumerate}
\end{defn}

We define the game $\mathcal{G}_{2}$ as an extended pebble game where \textbf{Prover} has another option than putting pebbles and asking queries; he can backtrack the game record, bringing back a certain amount of partial matching he currently has. 
We manage the record of the game by trees rather than sequences as in Definition \ref{DefG1} so that the notion of \textit{obliviousness} in Definition \ref{obliviousstrategy} and the proof of Theorem \ref{ProofIsStrategy2} become easy to describe.
The precise definition of $\mathcal{G}_2$ is as follows (for an informal description of the game, see Remark \ref{informalDefG2}):

\begin{defn}\label{DefG2}
Let $n,C \in \NN$.
Suppose we are given an $(n,C)$-tree $T$.
 $\mathcal{G}_{2}(n,C,T)$ is the following game:
  
 \begin{enumerate}
  \item Played by two players. We call them in the same way as in Definition \ref{DefG1}.
  \item \deff{A possible position} is a map (or a partial labeling) $L$ such that:
  \begin{enumerate}
   \item $\dom(L) \subseteq T$.
   \item For $v \in \dom(L)$, $L(v)$ is of the form $(M,A_{1},\ldots,A_{height(v)})$, where $M$ is a partial matching between $(n+1)$-pigeons and $n$-holes with size $\leq |n|^{C}\times height(v)$, and each $A_{j} \in 2^{|n|^{C}}$.
 
   \item $\dom(L)$ is closed downwards under $\subseteq$.
   \item $v \subseteq w \in \dom(L)$, if $(M,\vec{A})=L(v)$ and $(M',\vec{A}')=L(w)$, then they satisfy $M \subseteq M'$ and $\vec{A} \subseteq \vec{A}'$.
  \end{enumerate}
   
   \quad For future convenience, we introduce the following notations:
   \begin{itemize}
    \item Set $c(L) := \max\dom(L)$.
    \item Let $\mathcal{P}$ be \deff{the set of all possible positions}.
   \end{itemize}
   
   \item \deff{The initial position} is $L_{0}$, where $\dom(L_{0})$ is the rooted tree of height $0$, that is, consists only of the root $\emptyset$, and $L_{0}(\emptyset)= (\emptyset)$.
   \item Now, we describe transitions between positions together with each player's options:
   suppose the current position is $L$, and $L(c(L)) = (M,\vec{A})$.
   
   \begin{enumerate}
   \item\label{query} First, \textbf{Prover} plays a subset $Q \subseteq P_{n} \dot\cup H_{n}$ of cardinality $\# Q \leq |n|^{C}$, and send it to \textbf{Delayer}.   
   \item\label{answer} \textbf{Delayer} plays a minimal matching $M' \in \mathcal{M}_{n}$ covering $Q$, and send it back to \textbf{Prover}.
   
   \item\label{nextmove} \textbf{Prover} plays a triple $\langle o,x,B \rangle$, where $o \in [3]$, $x \in [2^{|n|^{C}}]$ if $o=1$, $x \subsetneq c(L)$ if $o =2,3$, and $B \in [2^{|n|^{C}}]$. 
   \item Depending on $o$, the next position (or judgment of the winner) is determined as follows:
    
    \begin{enumerate}
     \item We first describe the case when $o=1$.
     If $c(L) * x \not \in T$, then the game ends, and \textbf{Prover} loses.
     Otherwise, extend $L$ to a map on $\dom(L) \cup \{c(L)*x\}$ by labeling the child $c(L)*x$ with $(M \cup M', \vec{A},B)$, and let $L'$ be the resulting partial labeling on $T'$.
     Note that $c(L)*x \not \in \dom(L)$ by definition of $c(L)$.
     
     If $M \perp M'$, then the game ends, and \textbf{Prover} wins.
     Otherwise, $L' \in \mathcal{P}$, and it is the next position.
     
    \item If $o=2$, proceed as follows.
     Let $c(L) = x*k*\sigma$ ($k \in [2^{|n|^{C}}]$, $\sigma \in [2^{|n|^{C}}]^{\leq C}$. Note that $\sigma$ may be empty).
     If $x*(k+1) \not\in T$, the play ends and \textbf{Prover} loses.
     
     \quad Consider the case when $x*(k+1) \in T$. 
     Note that $x*(k+1) \not \in \dom(L)$ by $x*k \subseteq \dom(L)$.
     Let $(M'',\vec{\alpha}) = L(x)$.
     Extend $L$ by labeling $x*(k+1)$ with $(M'' \cup M', \vec{\alpha},B)$, and let $L'$ be the resulting map.
     
     \quad If $M'' \perp M'$ as matchings, then the game ends, and \textbf{Prover} wins.
     Otherwise, $L' \in \mathcal{P}$, and it is the next position.
     
     \item Consider the case when $o=3$.
     Let $c(L) = x*k*\sigma$.
     
     \quad If $x*(k-1) \not\in \dom(L)$, then the game ends, and \textbf{Prover} loses.
     
     \quad Otherwise, let $l$ be the maximal extension of $x*(k-1)$ within $\dom(L)$. 
          If $l$ is a leaf of $T$, the game ends, and \textbf{Prover} loses.
     
     Otherwise, let $(M'',\vec{\alpha}) := L(l)$.
     Let 
     \[T':=(\dom(L) \setminus \{v \in \dom(L) \mid x*k \subseteq v\}) \cup \{l*1\}\]
     Note that $l*1 \in T$ by assumption on $T$, and hence $T' \subseteq T$.
     Furthermore, $l*1 \not \in \dom(L)$ by definition of $l$. 
     Now, set
     \[L':= L \restriction_{(\dom(L) \cap T')} \cup \{l*1 \mapsto (M'' \cup M', \vec{\alpha},B)\}\]
     
     \quad If $M'' \perp M'$ as matchings, the game ends, and \textbf{Prover} wins.
     Otherwise, $L' \in \mathcal{P}$, and it is the next position.     
    \end{enumerate}
    \end{enumerate}
 \end{enumerate}
\end{defn}

\begin{rmk}\label{informalDefG2}
The intended meaning of each item above is as follows:
\begin{enumerate}
 \item \textbf{Delayer} pretends to have a truth assignment falsifying $ontoPHP^{n+1}_{n}$, which is, of course, impossible, and \textbf{Prover} wants to disprove it by querying how certain pigeons and holes are mapped to each other.
 \item 
 Intuitively, $L$ is \textbf{Prover}'s partial record of a play of the pebble game $\mathcal{G}_1$.
 $T$ is a proof-tree $\mathcal{T}$ of a $2^{|n|^{C}}$-sized $LK^{*}_{2+\frac{1}{2}, C}$-derivation $(\mathcal{T},\mathcal{S})$ of $ontoPHP^{n+1}_{n}$.
 For $v \in \dom(L) \subseteq T$,
 $L(v)=(M,\vec{A})$ means $M$ is the state of pebbles played at $v$, and $\vec{A}$ is auxiliary information for \textbf{Prover} to decide the next move, which may carry at most $C|n|^{C}$ bits of information and is insufficient to code the whole play of the game $\mathcal{G}_{2}$.
 
 $c(L)$ stands for \deff{the current frontier} of the position $L$.  
 
 \item The play starts at the root with no pebbles on the board.
 \item The play proceeds in turn, starting from \textbf{Prover}.
 \begin{enumerate}
  \item \textbf{Prover} queries at most $|n|^{C}$-many pigeons and holes $Q$ and sends it to \textbf{Delayer}.
  \item \textbf{Delayer} answers which holes (resp. pigeons) are matched to the pigeons (resp. the holes) queried.
  \item Based on \textbf{Delayer}'s answer, \textbf{Prover} chooses from one of the three options indicated by $o \in [3]$. 
  Furthermore, he specifies auxiliary information $x$ and $B$ (described in the following) to determine the next position precisely.
  \item 
  \begin{enumerate}
   \item If $o=1$, climb up to the child $c(L)*x$ of $c(L)$ and add $M'$ to the matching.
   \textit{Prover} also remembers $B$, which can code $|{n}|^{O(1)}$ bits of information.
   \item If $o=2$, move backward to $x \subsetneq c(L)$, climb up to the immediate right child of $x$, adding $M'$ to the matching, forgetting the differences of matching and $\vec{A}$ between $x$ and $c(L)$, but remembering $B$.
   Note that we keep a record of the labels of the vertices from $x$ up to $c(L)$.
   \item If $o=3$, move backward to $x \subsetneq c(L)$, climb up to the immediate left path from $x$, adding $M'$ to the matching and remembering $B$.
   This amounts to backtracking the game record, bringing back the information $M'$ obtained at the current position.
   Note that we erase the record of the labels of the vertices between $x$ and $c(L)$, including $c(L)$'s but not $x$'s.
  \end{enumerate}
 \end{enumerate}
\end{enumerate}
In any case, if \textbf{Prover} cannot play within the tree $T$, \textbf{Prover} loses.
If \textbf{Delayer}'s answer $M'$ contradicts the matching $M$ labeled at the parent of the subcedent $c(L)$, \textbf{Prover} wins.  
\end{rmk}


First, we observe that the game above is determined.
Towards it, we introduce the following binary relation:
\begin{defn}\label{finite well-order}
Let $T,T'$ be trees.
We write $T \prec T'$ if and only if there exists $w \in T'$ such that
\[T_{< w} = T'_{< w} \ \& \ w \not \in T.\]
Here, for a vertex $v$ of a tree $S$, $S_{<v}$ denotes $\{s \in S \mid s < v\}$.
Note that $T \prec T'$ implies $T \neq T'$.
Set $T \preceq T' :\Leftrightarrow T=T' \ \mbox{or} \ T \prec T'$.
\end{defn}

\begin{lemma}
$\preceq$ is a linear order on the set of finite trees.
\end{lemma}

\begin{proof}
First, $\preceq$ is antisymmetric.
Towards a contradiction, suppose $T \prec U$ and $U \prec T$. 
Let $u \in U$ witness $T \prec U$ and $t \in T$ witness $U \prec T$.
If $u \leq t$, then it contradicts the two assumptions $U_{< t} = T_{< t}$ and $u \not \in T$.
Similar for $t\leq u$, and it is absurd. 

Besides, $\prec$ is transitive. 
Suppose $S \prec T$ and $T \prec U$. Let $t \in T$ witness $S \prec T$ and $u \in U$ witness $T \prec U$.
If $t \leq u$, then $S_{<t} = T_{<t} = U_{<t}$, $t \not \in S$, and $t \in U$ because $T_{<u}=U_{<u}$ and $u \in U$.
If $t >u$, $S_{<u} = T_{<u}=U_{<u}$, and $u \not \in T$ together with $S_{<t}=T_{<t}$ implies $u \not\in S$.

Lastly, $\prec$ satisfies a trichotomy: if $T \neq U$, then $T \prec U$ or $U \prec T$.
Indeed, by finiteness of $T$ and $U$, take the maximum $w \in T \cup U$ such that $T_{<w}=U_{<w}$.
Note that $T_{<\emptyset} = U_{<\emptyset}$ trivially. 
Towards a contradiction, suppose $w \in T \cap U$. 
Then, since $T \neq U$ and $T_{<w}=U_{<w}$, there exists $v \in T \cup U$ such that $w < v$.
Take the least such $v$. 
Then
\[ T_{<v} = T_{<w} \cup \{w\} = U_{<w} \cup \{w\} = U_{<v},\]
contradicting the maximality of $w$.
Therefore, $w$ witnesses $T \prec U$ or $U \prec T$.

\end{proof}

\begin{rmk}
    If we restrict the order $\prec$ to finite trees of height $\leq h$ and $< b$-branching, we can embed the opposite of $\prec$ into $[b^{b^{h+1}}]$ via Cantor normal form with the base $b$ as follows.
    Given a finite tree $T$ of height $\leq h$ and $< b$-branching, set the ordinal $O_1(T):=\sum_{v \in L(T)} b^{h-height(v)}$, where $L(T)$ is the set of all the leaves in $T$.
    Furthermore, let $\bar{T}$ be the extension of $T$ obtained by adding the verteces of the form $v*k$ where $v \in T \setminus L(T)$ and $k \in [b-2]$. 
    Then define $O_2(T):=\sum_{v\in L(\bar{T})} b^{O_1((\bar{T})_{> v})}$.
    Here, $(\bar{T})_{>v}$ is the finite tree isomorphic to 
    \[ \{w \in \bar{T} \mid w = \emptyset \quad \mbox{or} \quad v <_{lex} w\}.\]
\end{rmk}

\begin{lemma}\label{monotone}
If a position $L$ transitions to $L'$ in a play of $\mathcal{G}_{2}(n,C,T)$, then $\dom(L) \prec \dom(L')$.

\end{lemma}

\begin{proof}

Suppose $\langle o,x,B \rangle$ is \textbf{Prover}'s option at $L$. 
We split cases:
\begin{enumerate}
 \item if $o=1$, then $c(L)*x$ witnesses $\dom(L) \prec \dom(L')$.
 \item if $o=2$, then $x*(k+1)$ witnesses $\dom(L) \prec \dom(L')$, where $x*k \subseteq c(L)$.
 \item if $o=3$, then $c(L')$ witnesses $\dom(L) \prec \dom(L')$.
 Note that $c(L')=l*1$, where $l$ is the maximum leaf of $\dom(L)$ extending $x*(k-1)$. 
 Here, $x*k \subseteq c(L)$ again. 
\end{enumerate}

\end{proof}

\begin{cor}\label{determinacy}
For any $n,C \in \NN$ and an $(n,C)$-tree $T$, $\mathcal{G}_{2}(n,C,T)$ ends within $2^{2^{(C+1)|n|^{C}}}$-steps, determining the winner.
\end{cor}

\begin{proof}
When the game ends, it always determines who is the winner. 
Therefore, it suffices to show that the game ends with at most $2^{2^{(C+1)|n|^{C}}}$-many transitions.
Since there are at most $2^{(2^{|n|^{C}})^{C+1}}$-many rooted trees of height $\leq C$ and $\leq 2^{|n|^{C}}$-branching as subsets of $T$, the result follows from Lemma \ref{monotone}.
\end{proof}

Note that, without any restriction, \textbf{Prover} easily wins $\mathcal{G}_{2}(n,C,T)$ for reasonable parameters:

\begin{prop}
Suppose $n,C \geq 2$, $2^{|n|^{C}} > n$.
Let $T$ be the $(n,C)$-tree as follows:
\[T:=[n+1]^{\leq 1} \cup \{i*1 \mid i \in [n+1]\}.\]

 Then there exists \deff{a winning strategy of} \textbf{Prover} \deff{for $\mathcal{G}_{2}(n,C,T)$}, that is, there is a pair $(f_{1},f_{2})$ of functions such that:
 \begin{enumerate}
  \item $f_{1} \colon \mathcal{P} \rightarrow (2^{P_{n} \dot\cup H_{n}})_{\leq |n|^{C}}$.
  \item $f_{2} \colon \mathcal{P} \times \mathcal{M}_{n} \rightarrow [3] \times [2^{|n|^{C}}]^{\leq C} \times [2^{|n|^{C}}]$, and $\langle o,x,B \rangle = f_{2}(L,M')$ satisfies $x \in [2^{|n|^{C}}]$ if $o=1$ and $x \subsetneq c(L)$ otherwise.
  \item If \textbf{Prover} keeps playing $Q=f_{1}(L)$ at stage \ref{query} in Definition \ref{DefG2} and, receiving \textbf{Delayer}'s answer $M'$ at stage \ref{answer}, plays $\langle o,x,B \rangle = f_{2}(L,M')$ at stage \ref{nextmove}, then \textbf{Prover} eventually wins $\mathcal{G}_{2}(n,C,T)$ regardless \textbf{Delayer}'s moves.
 \end{enumerate}
\end{prop}

\begin{proof}
Let $P_{n} = \{p_{1}, \ldots, p_{n+1}\}$.
Informally, \textbf{Prover}'s winning strategy is as follows:
\begin{itemize}
 \item Taking option $o=1$, ramify at the root asking the query $\{p_{1}\}$.
   The root will have $(n+1)$-many children.
 \item One by one, taking option $o=2$, ask queries $\{p_{2}\}, \ldots, \{p_{n+1}\}$. There are enough vertices to do it.
 \item Because of the pigeonhole principle, there will be $a<a'$ such that \textbf{Delayer}'s answers $p_{a} \mapsto h$ and $p_{a'} \mapsto h'$ share the same hole: $h=h'$.
 Once detecting the pair, repeatedly take option $o=3$ and bring back the information $p_{a'} \mapsto h$ to the vertex storing $p_{a} \mapsto h$.
\end{itemize}

Now, we give a formal presentation.

For the initial position, set 
\[f_{1}(L_{0}) := \{p_{1}\},\ f_{2}(L_{0},M) := \langle 1,1, * \rangle,\]
where $*$ can be any.

Let $L$ be a position with $height(c(L))=1$.
If $L(c(L))$ does not contradict $L(l)$ for any leaf $l$ of $\dom(L)$, set
\[f_{1}(L) := \{p_{a+1}\}\]
where $a:=c(L) \in [n+1]$.

Let $M'$ be a partial matching of the form $\{p_{a+1} \mapsto h\}$.
Set
\[f_{2} (L,M') := \langle 2, \emptyset, * \rangle.\]

Let $L$ be a position with $height(c(L)) \geq 1$.
 If there exists a leaf $l$ of $\dom(L)$ such that $L(l) \perp L(c(L))$, then let $a \in [n+1]$ be the maximum index such that $\{p_{a} \mapsto h\} \subseteq L(c(L))$ contradicts $L(l)$.
Set
\[f_{1}(L) := \{p_{a}\}\]
and 
\[f_{2} (L,\{p_{a} \mapsto h\}) := \langle 3, \emptyset, * \rangle.\]

Regardless values $f_{1}(L)$ and $f_{2}(L,M')$ for other $L$ and $M'$, $(f_{1},f_{2})$ is a winning strategy for \textbf{Prover} because of the assumption on the parameters.
\end{proof}

\begin{rmk}
We can see that the part $\vec{A}$ of the label of positions $L \in \mathcal{P}$ does not play any role in the proof above.
Actually, $\vec{A}$ is not essential at all for our use of $\mathcal{G}_{2}$.
See Appendix \ref{Auxiliary info not needed}.
However, the auxiliary information $\vec{A}$ is helpful to describe the obliviousness in the Definition \ref{obliviousstrategy} and in the proof of Theorem \ref{ProofIsStrategy2}.
\end{rmk}

Thus, to obtain a nontrivial and possibly useful notion, we must consider a specific limited class of \textbf{Prover}'s strategies.
The following ``obliviousness'' is our suggestion (roughly speaking, it restricts \textbf{Prover} to refer the information of $c(L)$ and $L(c(L))$ only):

\begin{defn}\label{obliviousstrategy}
Let $n,C \in \NN$.
Suppose $T$ is an $(n,C)$-tree.
\deff{An oblivious strategy of} \textbf{Prover} \deff{for $\mathcal{G}_{2}(n,C,T)$} is a pair $(f_{1},f_{2})$ of functions such that:
\begin{itemize}
 \item $f_{1} \colon \bigcup_{l=0}^{C}\left([2^{|n|^{C}}]^{l} \times \mathcal{M}_{n} \times [2^{|n|^{C}}]^{l} \right)\rightarrow (2^{P_{n} \dot\cup H_{n}})_{\leq |n|^{C}}$.
 \item $f_{2} \colon \bigcup_{l=0}^{C}\left([2^{|n|^{C}}]^{l} \times \mathcal{M}_{n} \times [2^{|n|^{C}}]^{l} \times \mathcal{M}_{n}\right) \rightarrow [3] \times [2^{|n|^{C}}]^{\leq C} \times [2^{|n|^{C}}]$, and $\langle o,x,B \rangle = f_{2}(v,M,\vec{A},M')$ satisfies $x \in [2^{|n|^{C}}]$ if $o=1$ and $x \subsetneq v$ otherwise.
\end{itemize}
\end{defn}

An important point of Theorem \ref{ProofIsStrategy2} is that, while \textbf{Prover} is restricted to oblivious strategies, \textbf{Delayer} has no restriction; she can see the whole position $L$ and make a decision:
\begin{defn}\label{Delayer's strategy}
\deff{A strategy of} \textbf{Delayer} \deff{for $\mathcal{G}_{2}(n,C,T)$} is a function $g$ such that:
\begin{itemize}
 \item $g \colon \mathcal{P} \times (2^{P_{n} \dot\cup H_{n}})_{\leq |n|^{C}} \rightarrow \mathcal{M}_{n}$, and $g(L,Q)$ is a minimal partial matching covering $Q$ for each $(L,Q)$ in the domain.
\end{itemize}

\end{defn}

\begin{defn}\label{strategybeatsanother}
 Let $g$ be \textbf{Delayer}'s strategy for $\mathcal{G}_{2}(n,C,T)$, and $(f_{1},f_{2})$ be an oblivious \textbf{Prover}'s strategy for $\mathcal{G}_{2}(n,C,T)$.
 \deff{$(f_{1},f_{2})$ beats $g$} if and only if \textbf{Delayer} wins $\mathcal{G}_{2}(n,C,T)$ if the two players play in the following way:
 \begin{enumerate}
  \item \textbf{Prover} plays $Q=f_{1}(v,M,\vec{A})$ at stage \ref{query} in Definition \ref{DefG2}, where $v:=c(L), (M,\vec{A}) :=L(v)$.
  \item \textbf{Delayer} answers $M'=g(L,Q)$ at stage \ref{answer}.
  \item \textbf{Prover} then plays $\langle o,x,B \rangle = f_{2}(v,M,\vec{A},M')$ at stage \ref{nextmove}.
 \end{enumerate}
\end{defn}

Now, we prove a counterpart of Proposition \ref{ProofIsStrategy1} for $\mathcal{G}_{2}$.
The proof can be regarded as a transposition of Buss's witnessing argument (\cite{Buss}, \cite{clearwitnessing}).

\begin{thm}\label{ProofIsStrategy2}
Suppose $ontoPHP^{n+1}_{n}$ has $2^{|n|^{O(1)}}$-sized $LK^{*}_{2+\frac{1}{2},O(1)}$-proofs.
Then, there exists $C>0$ such that, for any sufficiently large $n$, there exists an $(n,C)$-tree $T$ such that \textbf{Prover} has an oblivious strategy $(f_{1}, f_{2})$ for $\mathcal{G}_{2}(n,C,T)$ which beats arbitrary \textbf{Delayer}'s strategy $g$.
\end{thm}

\begin{proof}
 Our approach is similar to the sketch given in the last section.
 The main differences are:
 \begin{itemize}
  \item Since we treat $LK^{*}_{2+\frac{1}{2},O(1)}$-proofs this time, the witness function $W[P]$ will be defined on not only the $s\Pi_{1}$-formulae but also the $p\Pi_{2}$-formulae.
  It will be represented by $\vec{A}$ in the following argument.
  \item When we consider $\bigvee$-Rule: 
   \begin{prooftree}
 \AxiomC{$\Gamma, \varphi_{i_{0}}$} \RightLabel{\quad (where $\bigvee_{i=1}^{I}\varphi_{i} \in \Gamma$, $1\leq i_{0} \leq I$, $\varphi_{i_{0}} \not \in \Gamma$)}
  \UnaryInfC{$\Gamma$}
 \end{prooftree}
$\bigvee_{i=1}^{I}\varphi_{i}$ is either in $ontoPHP^{n+1}_{n}$ or $p\Sigma_{1}$ or $p\Sigma_{2}$.
In the first case, the arguments are the same.
In the last case, we argue as if $\varphi_{i_{0}}$ was a satisfied $p\Sigma_{0}$-formula in the case when $\bigvee_{i=1}^{I}\varphi_{i}$ was $p\Sigma_{1}$ in the sketch.
If $\bigvee_{i=1}^{I}\varphi_{i}$ is $p\Sigma_{1}$, then \textbf{Prover} makes a query to decide $\varphi_{i_{0}}$, which is $p\Sigma_{0}$, and if it is falsified, he just climbs up the proof tree, but if it is satisfied, then he backtracks to the point where $\bigvee$-Rule is applied to eliminate $\bigwedge_{i=0}^{k}\overline{\varphi_{i}}$.
Now he has a counterexample $i_{0}$ for it and restarts the play from the point, climbing up the $\bigvee$-Rule which he once went through assuming $\bigwedge_{i=0}^{k}\overline{\varphi_{i}}$ was true.
 \end{itemize} 
 
  Now, we present a precise proof.
  By assumption, there exists $C>0$ such that, for any sufficiently large $n$, there exists an $LK^{*}_{2+\frac{1}{2},C}$-derivation $\pi=(\mathcal{T},\mathcal{S})$ of $ontoPHP^{n+1}_{n}$ satisfying $|\pi| \leq 2^{|n|^{C}}$.
 By an isomorphism, we may assume that $\mathcal{T}$ is an $(n,C)$-tree.
 Note that each $\varphi$ appearing in $\pi$ is $s\Sigma^{2^{|n|^{C}},|n|^{C}}_{2}$ or $s\Pi^{2^{|n|^{C}},|n|^{C}}_{2}$ by Remark \ref{subformulaproperty}.
 For readability, we omit the superscripts and write $s\Pi_{i}$, $s\Sigma_{i}$, $p\Pi_{i}$ and so on below.
 
 We extract an oblivious winning strategy of \textbf{Prover} for $\mathcal{G}_{2}(n,C,\mathcal{T})$ from this $\pi$.
 The strategy sticks to play so that, for each position $L$,

 \begin{itemize}
 \item Given $v \in \dom(L)$, let $L(v) = (M,A_{1},\ldots,A_{height(v)})$. 
 Note that $height(v) \leq C$ by our formalization of $LK_{2+\frac{1}{2},C}^{*}$.
 Let $\varphi \in \mathcal{S}(v)$.
  
According to the complexity of $\varphi$, the following hold:
 \begin{enumerate}
  \item If $\varphi$ is $s\Sigma_{0}$, then $M \Vdash \overline{\varphi}$.
  \item Consider the case when $\varphi$ is $p\Pi_{2}$. 
  Let $\varphi = \bigwedge_{e=1}^{E}\psi_{e}$, where each $\psi_{e}$ is $p\Sigma_{1}$.
  It must be eliminated by $p\Sigma_{2}$-Induction along the path from $v$ to the root in $\pi$. 
  Let $w \subseteq v$ be the vertex such that $\mathcal{S}(w)$ is derived by $p\Sigma_{2}$-Induction and $\varphi$ is eliminated there.
  Then, $A_{height(w)+1} \leq k$.

 \item\label{counterexample} Consider the case when $\varphi$ is $p\Pi_{1}$.
 Let $\varphi = \bigwedge_{e=1}^{E}\psi_{e}$, where each $\psi_{e}$ is $s\Sigma_{0}$.
 $\varphi$ must be eliminated along the path from $v$ to the root in $\pi$, and the derivation rule must be $\bigvee$-Rule.
 
 Let $u \subseteq v$ be the vertex such that $\mathcal{S}(u)=\Gamma$ is derived by:
    \begin{prooftree}
 \AxiomC{$\Gamma, \varphi_{i_{0}}$} \RightLabel{\quad (where $\bigvee_{i=1}^{I}\varphi_{i} \in \Gamma$, $\varphi = \varphi_{i_{0}}$, $1\leq i_{0} \leq I$, $\varphi \not \in \Gamma$)}
  \UnaryInfC{$\Gamma$}
 \end{prooftree}
 
 Since every formula appearing in $\pi$ is $s\Sigma_{2}$ or $s\Pi_{2}$, we have that $\bigvee_{i=1}^{I}\varphi_{i}$ is $p\Sigma_{2}$, and therefore it must be an element of $ontoPHP^{n+1}_{n}$ or eliminated by $p\Sigma_{2}$-Induction along the path from $u$ to the root in $\pi$.
 
 In the former case, if 
 \begin{align*}
 \bigvee_{i=1}^{I}\varphi_{i}&= \bigvee_{p \in [n+1]} \bigwedge_{h \in [n]} \lnot r_{ph}, \\
 \varphi&= \bigwedge_{h \in [n]} \lnot r_{i_{0}h},
 \end{align*}
  then $A:=A_{height(u)+1} \in [n]$.
 Moreover, $M \Vdash \overline{\lnot r_{i_{0}A}}$.
 If 
 \begin{align*}
 \bigvee_{i=1}^{I}\varphi_{i} &= \bigvee_{h \in [n]} \bigwedge_{p \in [n+1]} \lnot r_{ph},\\
 \varphi&= \bigwedge_{p \in [n+1]} \lnot r_{pi_{0}},
 \end{align*} 
 then $A:=A_{height(u)+1} \in [n+1]$.
 Furthermore, $M \Vdash \overline{\lnot r_{A i_{0}}}$.

 In the latter case, then, $A:=A_{height(u)+1} \leq I$, and $M \Vdash \overline{\psi_{A}}$.
 \end{enumerate}
  \textit{Intuitively, according to the strategy}, \textbf{Prover} \textit{always stores a counterexample for each $p\Pi_{1}$-formula $\varphi \in \mathcal{S}(v)$ and a candidate of a counterexample for each $p\Pi_{2}$-formula, under $\rho^{M}$, in $\vec{A}$}.

 \item For each $w \in \dom(L)$, if $\mathcal{S}(w)=\Gamma$ is derived by $p\Sigma_{2}$-Induction:
     \begin{prooftree}
 \AxiomC{$\Gamma, \varphi_{1}$}
 \AxiomC{$\Gamma, \overline{\varphi_{1}},\varphi_{1}$}
 \AxiomC{$\cdots$}
 \AxiomC{$\Gamma, \overline{\varphi_{m-1}},\varphi_{m}$}
  \AxiomC{$\Gamma, \overline{\varphi_{m}}$}
  \QuinaryInfC{$\Gamma$}
 \end{prooftree}
 then $w*(k+1) \in \dom(L)$ implies $w*k \in \dom(L)$ (note that $k \in [m]$).
 Furthermore, if $w*(k+1) \in \dom(L)$, then, for the maximum leaf $l \in \dom(L)$ properly extending $w*k$, $\mathcal{S}(l)=\Delta$ is derived by $\bigvee$-Rule
    \begin{prooftree}
 \AxiomC{$\Delta, \psi_{i_{0}}$} \RightLabel{\quad ($\bigvee_{i=1}^{I}\psi_{i} \in \Delta$, $1\leq i_{0} \leq I$)}
  \UnaryInfC{$\Delta$}
 \end{prooftree}
  and the principal formula $\bigvee_{i=1}^{I}\psi_{i}$ is the $\varphi_{k}$ eliminated from $\mathcal{S}(w*k)$ deriving $\mathcal{S}(w)$ (therefore each $\psi_{i}$ is $p\Pi_{1}$).
  Furthermore, if $L(w*(k+1))=(M,\vec{A})$, then $A_{height(w*(k+1))} = i_{0}$.
  
  \item For each non-leaf $w \in \dom(L)$, if $\mathcal{S}(w)=\Gamma$ is derived by $\bigwedge$-Rule:
      \begin{prooftree}
 \AxiomC{$\Gamma, \varphi_{1}$}
 \AxiomC{$\Gamma, \varphi_{2}$}
 \AxiomC{$\cdots$}
 \AxiomC{$\Gamma, \varphi_{I}$}
  \RightLabel{\quad (where $\bigwedge_{i=1}^{I}\varphi_{i} \in \Gamma$, and $\varphi_{i} \not \in \Gamma$ for each $i \in [I]$)}
  \QuaternaryInfC{$\Gamma$}
 \end{prooftree} 
 then $w$ has a unique child in $\dom(L)$.
 If $\bigwedge_{i=1}^{I}\varphi_{i}$ is $p\Pi_{1}$ or $p\Pi_{2}$ in particular, then, letting $A$ be the counterexample (or the candidate of a counterexample)  $A_{height(u)+1}$ for it described in (\ref{counterexample}) at the first item, the unique child is $w*A$.
 
 
 \end{itemize}
 
 Now, we describe the strategy explicitly. 
 Suppose we are at a position $L$ satisfying the above conditions.
 Note that the initial position trivially satisfies them.
 Let 
 \[v:=c(L),\ (M,\vec{A}):=L(v).\]
 
 We describe $f_{1}(v,M,\vec{A})$ and $f_{2}(v,M,\vec{A},M')$ splitting cases by the rule deriving $\mathcal{S}(v)$ in $\pi$.
 Note that, in each case, the conditions above remain satisfied. 
 $\mathcal{S}(v)$ is not an Initial cedent since if it was the case, then the literals $x$ and $\bar{x}$ in $\mathcal{S}(v)$ should be falsified by $M$, which is absurd.

 \begin{enumerate}
 \item The case when $\mathcal{S}(v)=\Gamma$ is derived by $\lor$-Rule: 
   \begin{prooftree}
 \AxiomC{$\Gamma, \varphi_{i_{0}}$} \RightLabel{\quad (where $\varphi_{1} \lor \varphi_{2} \in \Gamma$, $i_{0}=1$ or $i_{0}=2$, $\varphi_{i_{0}} \not \in \Gamma$)}
  \UnaryInfC{$\Gamma$}
 \end{prooftree}
 Since $\varphi_{1} \lor \varphi_{2}$ is $s\Sigma_{2}$ or $s\Pi_{2}$, it must be $s\Sigma_{0}$.
Set $f_{1}(v,M,\vec{A}):=\emptyset$. 
 \textbf{Delayer}'s answer should be $\emptyset$, and so set
 \[f_{2}(v,M,\vec{A},\emptyset):= \langle 1, 1, * \rangle,\]
 where $*$ can be any number in $[2^{|n|^{C}}]$. 

  \item The case when $\mathcal{S}(v)=\Gamma$ is derived by Trivial Cut or $\land$-rule:
  similar as above.
  \item The case when $\mathcal{S}(v)=\Gamma$ is derived by $p\Sigma_{2}$-Induction:
    \begin{prooftree}
 \AxiomC{$\Gamma, \varphi_{1}$}
 \AxiomC{$\Gamma, \overline{\varphi_{1}},\varphi_{2}$}
 \AxiomC{$\cdots$}
 \AxiomC{$\Gamma, \overline{\varphi_{I-1}},\varphi_{I}$}
  \AxiomC{$\Gamma, \overline{\varphi_{I}}$}
  \RightLabel{\quad (where each $\varphi_{i}$ is $p\Sigma_{2}$)}
  \QuinaryInfC{$\Gamma$}
 \end{prooftree}
 
 Set $f_{1}(v,M,\vec{A}):=\emptyset$. 
 \textbf{Delayer}'s answer should be $\emptyset$, and so set
 \[f_{2}(v,M,\vec{A},\emptyset):= \langle 1, 1, * \rangle,\]
 where $*$ can be any number in $[2^{|n|^{C}}]$. 
 
  \item The case when $\mathcal{S}(v)=\Gamma$ is derived by $\bigvee$-Rule:
     \begin{prooftree}
 \AxiomC{$\Gamma, \varphi_{i_{0}}$} \RightLabel{\quad (where $\bigvee_{i=1}^{I}\varphi_{i} \in \Gamma$, $1\leq i_{0} \leq I$, $\varphi_{i_{0}} \not\in\Gamma$)}
  \UnaryInfC{$\Gamma$}
 \end{prooftree}
 
 We further split cases according to the complexity of $\bigvee_{i=1}^{I}\varphi_{i}$.
  \begin{enumerate}
   \item First we consider the case when $\bigvee_{i=1}^{I}\varphi_{i} \in ontoPHP^{n+1}_{n}$. If
   \[\bigvee_{i=1}^{I}\varphi_{i} = \bigvee_{\substack{p \neq p^{\prime} \in [n+1],\\ h \in [n]}} (r_{ph} \land r_{p^{\prime}h}),\]
      let $i_{0}=\langle p,p^{\prime},h\rangle$.
   Then set $f_{1}(v,M,\vec{A}) := \{p,p^{\prime}\}$.
   \footnote{Actually, we may ask just the hole $h$ here, but if we were to consider the case when we deal with the injective pigeonhole principle $injPHP^{n+1}_{n}$ instead of $ontoPHP^{n+1}_{n}$ and would like to ask only pigeon-queries in the course of the corresponding game, the query $\{p,p^{\prime}\}$ is the right one. (note that we do not have the last two formulae of $ontoPHP^{n+1}_{n}$ in that case.)}

   Let $M'$ be \textbf{Delayer}'s answer.
   $M'$ always falsifies $\varphi_{i_{0}}$, so set 
   \[f_{2}(v,M,\vec{A},M') := \langle 1, 1,*\rangle.\]
   
   \quad If 
   \[\bigvee_{i=1}^{I}\varphi_{i}= \bigvee_{p \in [n+1]}\bigwedge_{h \in [n]} \lnot r_{ph}\]
   and $\varphi_{i_{0}} = \bigwedge_{h \in [n]} \lnot r_{i_{0}h}$,   
then set 
\[f_{1}(v,M,\vec{A}) := \{i_{0}\}.\]

   Let $M'=\{i_{0} \mapsto h_{0}\}$ be \textbf{Delayer}'s answer.
   $M'$ together with the witness $j_{0}$ falsifies $\varphi_{i_{0}}$, so set 
   \[f_{2}(v,M,\vec{A},M') := \langle 1, 1, h_{0}\rangle.\]
   
   If $\bigvee_{i=1}^{I}\varphi_{i}$ is one of other elements in $ontoPHP^{n+1}_{n}$, change the roles of pigeons and holes and play similarly as above.
   
   Note that \textbf{Prover} does not lose by choosing $o=1$ since $v*1 \in \mathcal{T}$.

   \item Else if $\bigvee_{i=1}^{I}\varphi_{i}$ is $p\Sigma_{2}$ and $\varphi_{i_{0}}$ is $p\Pi_{1}$, then, since $\bigvee_{i=1}^{I}\varphi_{i} \not \in ontoPHP^{n+1}_{n}$, there uniquely exists $w*e \subseteq v$ such that $\mathcal{S}(w)$ is derived by $p\Sigma_{2}$-Induction, and $\bigvee_{i=1}^{I}\varphi_{i}$ is eliminated there.
   By definitions,
   \[w*(e+1) \in \mathcal{T} \ \& \ \bigwedge_{i=1}^{I}\overline{\varphi_{i}} \in \mathcal{S}(w*(e+1)).\]
    Based on this, set 
    \[f_{1}(v,M,\vec{A}):= \emptyset, f_{2}(v,M,\vec{A},\emptyset) := \langle 2, w, i_{0}\rangle.\]
    Note that \textbf{Prover} does not lose by choosing $o=2$ since $w*(e+1) \in \mathcal{T}$.
    \item Else if $\bigvee_{i=1}^{I}\varphi_{i}$ is $p\Sigma_{1}$ and $\varphi_{i_{0}}$ is $s\Sigma_{0}$, then, let $f_{1}(v,M,\vec{A})$ be the set of pigeons\footnote{Similarly to before, we may also query holes here.} appearing as indices of variables occuring in $\varphi_{i_{0}}$, which is at most cardinality $\leq |n|^{C}$.
    
    Let $M'$ be \textbf{Delayer}'s answer. $M'$ decides $\varphi_{i_{0}}$, that is, 
    \[M' \Vdash \varphi_{i_{0}} \ \mbox{or} \ M' \Vdash \overline{\varphi_{i_{0}}}.\]
    If $M' \Vdash \overline{\varphi_{i_{0}}}$, then set 
    \[f_{2}(v,M,\vec{A},M') := \langle 1,1,* \rangle.\]
    If $M' \Vdash \varphi_{i_{0}}$, then, since $\bigvee_{i=1}^{I}\varphi_{i} \not \in ontoPHP^{n+1}_{n}$, there exists $w*e \subseteq v$ such that $\mathcal{S}(w)=\Theta$ is derived by $p\Sigma_{2}$-Induction:
        \begin{prooftree}
 \AxiomC{$\Theta, \Phi_{1}$}
 \AxiomC{$\Theta, \overline{\Phi_{1}},\Phi_{2}$}
 \AxiomC{$\cdots$}
 \AxiomC{$\Theta, \overline{\Phi_{E-1}},\Phi_{E}$}
  \AxiomC{$\Theta, \overline{\Phi_{E}}$}
  \RightLabel{\quad (where each $\Phi_{e}$ is $p\Sigma_{2}$)}
  \QuinaryInfC{$\Theta$}
  \end{prooftree}
    $e \geq 2$, and $\overline{\Phi_{e-1}}=\bigwedge_{\alpha}\bigvee_{\beta=0}^{B_{\alpha}}\psi_{\alpha\beta}$ has $\bigvee_{i=1}^{I}\varphi_{i}$ as one of its conjuncts, that is, there is $\alpha_{0}$ such that $\bigvee_{i=1}^{I}\varphi_{i} = \bigvee_{\beta=0}^{B_{\alpha_{0}}}\psi_{\alpha_{0}\beta}$.
    Furthermore, there exists $u$ with $w*e \subseteq u \subsetneq v$ such that $\Lambda=\mathcal{S}(u)$ is derived by $\bigwedge$-Rule:
          \begin{prooftree}
 \AxiomC{$\Lambda, \bigvee_{\beta=0}^{B_{1}}\psi_{1\beta}$}
 \AxiomC{$\Lambda, \bigvee_{\beta=0}^{B_{2}}\psi_{2\beta}$}
 \AxiomC{$\cdots$}
 \AxiomC{$\Lambda, \bigvee_{\beta=0}^{B_{m}}\psi_{m\beta}$}
    \QuaternaryInfC{$\Lambda$}
 \end{prooftree} 
 where $\overline{\Phi_{e-1}}=\bigwedge_{\alpha}\bigvee_{\beta=0}^{B_{\alpha}}\psi_{\alpha\beta} \in \Lambda$
 and $u*\alpha_{0} \subseteq v$.
 We set 
    \[f_{2}(v,M,\vec{A},M') := \langle 3, w, i_{0} \rangle.\]
    
    The resulting next position again satisfies the assumptions of the strategy.
 Indeed, by assumption, the counterexample $A_{height(w)+1}$ for $\overline{\Phi_{e-1}}$ stored at $w*e$ is $\alpha_{0}$. 
 We consider the maximum leaf $l \in \dom{L}$ extending $w*(e-1)$.
 Note that $w*(e-1) \in \dom(L)$ by assumption. 
  $\mathcal{S}(l)=\Delta$ is derived by the following $\bigvee$-Rule:
       \begin{prooftree}
 \AxiomC{$\Delta, \bigwedge_{\beta=0}^{B_{\alpha_{0}}}\overline{\psi_{\alpha_{0}\beta}}$} \RightLabel{\quad (where $\Phi_{e-1}= \bigvee_{\alpha}\bigwedge_{\beta=0}^{B_{\alpha}}\overline{\psi_{\alpha\beta}} \in \Delta$)}
  \UnaryInfC{$\Delta$}
 \end{prooftree}
The following figure shows the whole situation (the notation $@v$ shows the corresponding vertices):
 
 \begin{prooftree}\label{whenbacktrack}
  \AxiomC{$\cdots$}
  
 \AxiomC{$\vdots$}
  \UnaryInfC{$\Delta, \bigwedge_{\beta=0}^{B_{\alpha_{0}}}\overline{\psi_{\alpha_{0}\beta}} \quad (@ l*1)$}
  \UnaryInfC{$\Delta \quad (@ l)$}
  \UnaryInfC{$\vdots$}
  \UnaryInfC{$\Theta, \overline{\Phi_{e-2}}, \Phi_{e-1} \quad (@ w*(e-1))$}

 \AxiomC{$\cdots$}
 
 \AxiomC{$\vdots$}
 \UnaryInfC{$\Gamma, \varphi_{i_{0}} \quad (@ v*1)$}
 \UnaryInfC{$\Gamma \quad (@ v)$}
 \UnaryInfC{$\vdots$}
 \UnaryInfC{$\Lambda, \bigvee_{\beta=0}^{B_{\alpha_{0}}}\psi_{\alpha_{0}\beta} \quad (@ u*\alpha_{0})$}
  \AxiomC{$\cdots$}
 \TrinaryInfC{$\Lambda \quad (@ u)$}
  \UnaryInfC{$\vdots$}
  \UnaryInfC{$\Theta, \overline{\Phi_{e-1}}, \Phi_{e} \quad (@ w*e)$}
  
  \AxiomC{$\cdots$}
  \QuaternaryInfC{$\Theta \quad (@ w)$}
  
 \end{prooftree}
 Here, we disregard $\overline{\Phi_{e-2}}$ when $e=2$.
 
    Hence, recalling $\bigwedge_{\beta=0}^{B_{\alpha_{0}}}\overline{\psi_{\alpha_{0}\beta}} = \bigwedge_{i=1}^{I}\overline{\varphi_{i}}$ and $M' \Vdash \overline{\varphi_{i_{0}}}$, the claim follows.
    
    \item Otherwise, $\bigvee_{i=1}^{I}\varphi_{i}$ is $p\Sigma_{0}$ and $M$ already falsifies it, so set 
    \[f_{1}(v,M,\vec{A}) := \emptyset, f_{2}(v,M,\vec{A},\emptyset) := \langle 1, 1, * \rangle.\]
    
  \end{enumerate}

  \item  The case when $\mathcal{S}(v)=\Gamma$ is derived by $\bigwedge$-Rule:
     \begin{prooftree}
 \AxiomC{$\Gamma, \varphi_{1}$}
 \AxiomC{$\Gamma, \varphi_{2}$}
 \AxiomC{$\cdots$}
 \AxiomC{$\Gamma, \varphi_{I}$}
  \RightLabel{\quad (where $\bigwedge_{i=1}^{I}\varphi_{i} \in \Gamma$, and $\varphi_{i} \not \in \Gamma$ for each $i \in [I]$)}
  \QuaternaryInfC{$\Gamma$}
 \end{prooftree}
 If $\bigwedge_{i=1}^{I}\varphi_{i}$ is $s\Sigma_{0}$, it is already falsified, hence there exists $k$ such that $M \Vdash \overline{\varphi_{k}}$. 
Let $k$ be the least one, and set
 \[f_{1}(v,M,\vec{A}) := \emptyset, f_{2}(v,M,\vec{A},\emptyset) := \langle 1,k,* \rangle.\]
 
 If $\bigwedge_{i=1}^{I}\varphi_{i}$ is $p\Pi_{1}$ or $p\Pi_{2}$,
 By assumption, a counterexample $A_{e}$ for it is stored in $\vec{A}$.
 Set 
 \[f_{1}(v,M,\vec{A}) = \emptyset, f_{2}(v,M,\vec{A},\emptyset) := \langle 1,A_{e},* \rangle,\]
Note that $M \Vdash \overline{\varphi_{A_{e}}}$ if $\bigwedge_{i=1}^{I}\varphi_{i}$ is $p\Pi_{1}$.  
 
 \end{enumerate}
 
 This completes the description of \textbf{Prover}'s strategy. 
 Since \textbf{Prover} can continue the play as long as \textbf{Delayer} does not make a contradiction, and $\dom(L)$ strictly increases in terms of $\prec$ by Lemma \ref{monotone}, if \textbf{Delayer} wins, $c(L)$ reaches to one of Initial cedents, which is absurd since an Initial cedent can never be falsified.
 Hence, \textbf{Prover}'s strategy above is a winning one.
\end{proof}

\section{Analysis of simplified $\mathcal{G}_{2}$}\label{Analysis of simplified G2}
\quad As far as we see, analysis of general $\mathcal{G}_{2}(n,C,T)$ is rather difficult.

Thus, in this section, we focus on $\mathcal{G}_{2}(n,C,T)$ of minimal height $C$ allowing backtracking, namely, $C=2$.

Even this case is hard to analyze, so we also assume that all the queries in the play are of size $1$.
Moreover, we further restrict \textbf{Prover}'s oblivious strategy $(f_{1},f_{2})$ to the following ones:
\begin{enumerate}
 \item Let $s \in [2^{|n|^{C}}]$.
 \item There exists a map 
 \[\mathcal{A} \colon (\mathcal{M}_{n})_{\leq 1} \rightarrow P_{n}\]
  such that, for $v \in [2^{|n|^{C}}]^{\leq 1}$, $f_{1}(v,M,A) = \{\mathcal{A}(M)\}$.
 \item Let $P_{n} = \{p_{1}, \ldots, p_{n+1}\}$. Let $p_{i} = \mathcal{A}(\emptyset)$.
 Then, for any $h \in H_{n}$,
 \[f_{2}(\emptyset, L_{0}(\emptyset), \{p_{i} \mapsto h\}) = \langle 1, 1, 1\rangle.\]
 \item For $v \in [2^{|n|^{C}}]^{1}$, $M,M' \in \mathcal{M}_{n}$ of size $1$, and $A \in [s]$, 
 \begin{align*}
  f_{2}(v,M,A,M') =
  \begin{cases}
   \langle 2,\emptyset, A+1 \rangle \quad &\mbox{(if $A<s$)}\\
   \langle 3,\emptyset, k \rangle \quad &\mbox{(if $A=s$, and $M'$ covers $p_{k}$)}
  \end{cases}.
  \end{align*}
 \item For $v \in [2^{|n|^{C}}]^{2}$, $M \in \mathcal{M}_{n}$ of size $2$, and $A_{1},A_{2} \in [s]$, if $A_{2} \in [n+1]$ and $M$ matches $p_{A_{2}}$ to $h \in H_{n}$, then 
 \begin{align*}
  f_{1}(v,M,A_{1},A_{2}) = \{p_{A_{2}}\},
 \end{align*}
 and 
 \begin{align*}
  f_{2}(v,M,A_{1},A_{2}, \{p_{A_{2}} \mapsto h\}) = \langle 3,\emptyset, A_{2} \rangle.
 \end{align*}
\end{enumerate}
Note that values of $f_{1}$ and $f_{2}$ for other inputs do not matter to determine the winner.
This very limited ``toycase'' of $\mathcal{G}_{2}(n,C,T)$ is still nontrivial since $s$ may be greater than $n$, which gives advantage to \textbf{Prover}. Actually, it turns out that the bound of $s$ above is not essential in this situation.

Below we present the mentioned restriction in a self-contained manner, with notation and terminology consistent with the previous sections.

\begin{defn}
    Fix an integer $n \geq 1$. We define the following \deff{game} (denoted as $\mathcal{G}_{2}(n)$) played by two players: \textbf{Prover} (denoted as \P and referred to as \textit{he}) and \textbf{Delayer} (denoted as \D and referred to as \textit{she}). 
    
    \D claims the existence of an injective function $f$ from the set $\{0, ..., n\}$ (denoted as $P_{n}$) into the set $\{0, ..., n-1\}$ (denoted as $H_{n}$). We think of $P_{n}$ as a set of \textit{pigeons} and $H_{n}$ as a set of \textit{holes}. 
    
    Through the course of the game \P asks \D questions of the form ``Where is the pigeon $p \in P_{n}$ being mapped to by $f$" (we denote such question as $p$ itself) and \D must answer with ``$p$ is being mapped to $h$ by $f$" for some $h \in H_{n}$ (we denote such answer as $h$ itself). These answers are then stored by \P in the form $(p, h)$.

    \PP's goal in the game is to force \D into giving two contradictory answers. This means \P wins iff \D answers $h$ to the question $p$, while \PP's list of records contains either $(p', h)$ or $(p, h')$ for $p' \neq p$ and $h' \neq h$.

    At each moment during the game, \P can only store at most 2 records in his list. This means that before asking a new question \P must make room for an upcoming answer from \DD. Moreover, \P must always remove the oldest record. 
    
    Finally, before the game starts, \P chooses an integer parameter $s \geq 1$ which \D sees, as well. As soon as the total number of the given answers in a game reaches $s$, the game stops, and the final answer of \D is compared to \textit{all} the records that have appeared through the whole game. If there is a single contradiction of the form $(p, h), (p', h)$ or $(p, h), (p, h')$ for $p' \neq p$ and $h' \neq h$ and with $p$ being the last question and $h$ being the last answer, \P wins, otherwise \D wins.  
\end{defn}

\begin{defn}
    A \textbf{\textit{strategy}} for \P is a pair $(s, F)$ of an integer $s \geq 1$ and a function $F$ which on input the empty set (the initial game state) or a single record $(p, h)$ (\PP's list after the removal of an older record) outputs a question $p'$ which \P should ask next. 
    
    We say \P plays a game according to the strategy $(s, F)$ iff his zeroth move is to claim $s$ as the game's initial parameter, his first move is to ask a question $F(\emptyset)$ and, finally, after the removal of an older record and being left with a single record $(p, h)$, if the game is not yet finished, \P asks $F((p, h))$.  
    
    From now on $F((p, h))$ will be denoted simply as $F(p, h)$.

    We say a strategy $(s, F)$ is \deff{winning} for \P iff, for plays in which \D never gives two contradictory answers in a row, any last answer of \D in a game where \P plays according to $(s, F)$ contradicts at least one prior answer of \D.    
\end{defn}

The ultimate goal is to show that for any $n \geq 3$, \P has no winning strategies. However, before this, we show two simple cases when \P can win.

\begin{prop}
    For $n \leq 2$ there is a winning strategy for \P with $s$ being any number $\geq n + 1$.
\end{prop}
\begin{proof}
    We start with the case $n = 2$. 
    
    Let $s$ be $3$. \P starts by asking $0 \in P_{2}$. Then, no matter what the answer is, \P asks $1 \in P_{2}$. Since \D cannot give two contradictory answers, after the second step, \PP's list should contain records $(0, h)$ and $(1, h')$ with $h \neq h'$. 
    
    \P removes the older record $(0, h)$ and asks $2 \in P_{2}$. Since $s = 3$, the game stops after \D answers that question. Such an answer is of the form $h''$, but since $H_{n}$ is of size 2 and $h \neq h' \in H_{n}$, it follows that $h'' = h$ or $h'' = h'$. Since we assumed \D never gives two contradictory answers in a row, it follows $h'' = h$ and so the record $(2, h'')$ contradicts the record $(0, h)$.

    For $s > 3$ it is enough to start as before, but after the 3rd answer, \P continues to ask $2 \in P_{2}$ throughout the rest of the game. Since \D cannot give two contradictory answers, she will be giving the same answer for the rest of the game and we have seen that the corresponding record $(2, h'')$ necessarily contradicts the first record $(0, h)$.

    For the remaining case $n = 1$, \P can win even before the game reaches $s$, since asking first $0 \in P_{1}$ and then $1 \in P_{1}$ necessary results in a contradiction.
\end{proof}

In the definition of the game we have assumed $P_{n} = \{0, ..., n\}$ and $H_{n} = \{0, ..., n-1\}$. For the proposition, we consider a more general case with the gap between the sizes of $P_{n}$ and $H_{n}$ much larger than 1.

\begin{prop}
    Let $P_{n}$ be $\{0, ..., 2^{n} - 1\}$ and $H_{n}$ be $\{0, ..., n-1\}$. Then, for any $n \geq 1$ \PP, has a winning strategy with $s$ being any parameter $\geq n + 1$.
\end{prop}
\begin{proof}
    Fix $s = n + 1$ as before. We let the set $P_{n}$ be labeled by all the subsets of $H_{n}$. The first question is $\emptyset$ (i.e., \P queries the hole for the pigeon labeled as $\emptyset \subset H_n$). 
    
    At step $i > 1$, assume that the record of \P is $(S, h)$, where $S \subseteq H_n$. In case $h \in S$, \P queries $S$ again. We can w.l.o.g. assume \D answers $h$; otherwise, \P wins immediately.

    In case $h \notin S$, \P queries $S \cup \{h\}$.

    We show that if \D ever answers $h$ to a query $S$, so that $h \in S$, \P wins the game. Assume this happens first time at step $i > 1$. Consider the sequence of records accumulated so far. It is of the form $(S_1, h_1), (S_2, h_2), \dots (S_i, h_i)$, where, by induction and minimality of $i$, we have that $S_{j + 1} = S_j \cup \{h_j\}$ and $S_{j + 1} \supset S_j$ for $j < i$. Moreover, $S_1 = \emptyset$. This all implies the existence of $j < i$ such that $h_j = h_i$, while $S_j \subset S_i$ and so, in particular, $S_j \neq S_i$. In other words, records $(S_j, h_j)$ and $(S_i, h_i)$ are inconsistent. Finally, since $h_i \in S_i$, it follows that all the subsequent queries of \P are equal to $S_i$, while \DD's answers to them are equal to $h_i$. In particular, the last query of \P in the game is $S_i$, and the subsequent answer of \D is $h_i$, which contradicts the record $(S_j, h_j)$.
    
    To finish the proof, it is enough to show that there must be a step $i \leq s$ at which \D answers $h$ such that $h \in S$. This holds since each time \D answers $h$, which is outside of the queried set $S$, the next set queried by \P is of strictly greater size than $S$. Since $H_n$ is of size $n$, the above can happen at most-$n$ times in total, while the game is played for $n + 1$-many rounds.
\end{proof}

We leave as an open question whether one can get the gap between the sizes of $P_{n}$ and $H_{n}$ smaller, even by a constant.

\begin{question}
    Do there exist integers $c \geq 2$ and $n_{0}, s_{0}$ such that, for all $n \geq n_{0}$ and $s \geq s_{0}$, \P has a winning strategy for a game played with $P_{n} = \{0, ..., 2^{n} - c\}$ and $H_{n} = \{0, ..., n-1\}$?
\end{question}

For the rest of the paper, we consider the case with $P_{n} = \{0, ..., n\}$ and $H_{n} = \{0, ..., n-1\}$ as in the original definition. We already know that \P can win in a game for $n \leq 2$. What we are going to show is that for any $n \geq 3$, \P cannot win. The proof can be viewed either as a proof-by-contradiction, where we show that no strategy can be a winning one, or as a construction of an \textit{anti-strategy} for \D given a fixed strategy for \PP. Here, anti-strategy is defined as a strategy for \D, which moreover depends on a chosen strategy for \PP.

\begin{prop}
\label{canonical anti-strategy}
    Let $n \geq 3$ and $(s, F)$ be a strategy for \PP. If $s \leq n$, then there is a winning anti-strategy for \DD. Moreover, such anti-strategy does not depend on $F$.
\end{prop}
\begin{proof}
    The anti-strategy is very simple. Being presented with a question $p$, \D answers $h$ so that
    \begin{enumerate}
        \item if \P has not asked $p$ before, then $h$ is any hole from $H_{n}$ which has not been used by \D yet; in case there is no such hole, \D ``gives up" and answers 0;
        \item if \P has asked $p$ before, then $h$ is the same as the one \D answered to $p$ for the first time. 
    \end{enumerate}
    Since $s \leq n$, \D will never run out of holes and so will never give up before the game ends. Notice that through the whole game, all answers of \D are \textit{compatible} (i.e. are not contradictory) with any previous answers. This shows that the $s$th answer is not contradictory, as well. 
\end{proof}

\begin{defn}
    Any anti-strategy for \D which satisfies conditions 1. and 2. from the above proof is called a \deff{canonical anti-strategy}.
\end{defn}

\begin{prop}
\label{loop}
    Let $n \geq 3$ and $(s, F)$ be a strategy for \PP. Assume \D is played according to a canonical anti-strategy and case 2. from the proof of \ref{canonical anti-strategy} occurs at step $i \leq \min\{s, n\}$, i.e., \P queries the pigeon $p$ which he has queried before. Then, there is a winning anti-strategy for \D against $(s', F)$, where $s'$ is arbitrary.
\end{prop}
\begin{proof}
    The claim follows immediately from the proof of \ref{canonical anti-strategy}. Namely, if the play is such that \D can satisfy conditions 1. and 2. of the canonical anti-strategy and is never forced to give up, then \D wins.

    As was noted, in case $s \leq n$, it is always possible for \D to carry out the chosen canonical anti-strategy without giving up, while it might not be possible if $s > n$.

    However, under the assumptions of the current proposition, we claim that \D is able to proceed without ever needing to give up. To see this, let $i \leq \min\{s, n\}$ be such that the sequence of records accumulated so far is of the form $(p_1, h_1), \dots, (p_i, h_i)$, with $p_i = p_j$ for $j < i$, and all $(p_k, h_k), (p_l, h_l)$ compatible for $k, l \leq i$. 
    
    It follows that $h_i = h_j$. Since $F$ depends solely on the last record, we derive that $p_{i + 1} = p_{j + 1}$, and $h_{i + 1} = h_{j + 1}$, since \D is playing according to a canonical anti-strategy.

    The claim now follows, since after the $i$th step, the game ``loops", i.e. the sequence of records accumulated by \P is of the form 
    \begin{align*}
        (p_1, h_1), \dots, (p_j, h_j), (p_{j + 1}, h_{j + 1}), \dots, (p_i, h_i), (p_{j + 1}, h_{j + 1}), \dots, (p_i, h_i), (p_{j + 1}, h_{j + 1}), \dots,
    \end{align*}
    where all the records are pairwise compatible.
\end{proof}

We now need a better way to represent \PP's strategies. This will help us to construct anti-strategies for \DD.

\begin{defn}
    Let $(s, F)$ be a strategy for \PP. We call $F$ a \deff{functional part} of $(s, F)$.
    
    A \deff{graph} associated with $F$ is a directed labeled multigraph (with loops allowed) $\mathcal{G}_{F} = (G_{F}, E_{F}, l_{F})$, defined as follows. The set $G_{F}$ of nodes equals $P_{n}$, and for each $p \in P_n$ and $h \in H_n$, there is exactly one edge from $p$ to $p'$ labeled as $h$ by $l_F$ such that $F(p, h) = p'$.

    Two edges $e, e'$ going from $p$ and $p'$, respectively, are \deff{incompatible} (or \deff{inconsistent}) iff either of the following conditions is satisfied
    \begin{itemize}
        \item $p = p'$ and $l_{F}(e) \neq l_{F}(e')$;
        \item $p \neq p'$ and $l_{F}(e) = l_{F}(e')$. 
    \end{itemize}

    We say that a node $p \in G_{F}$ is an \deff{initial node} iff $p = F(\emptyset)$. 
    
    A \deff{path} in $\mathcal{G}_{F}$ is defined as a directed walk in $\mathcal{G}_{F}$, which starts at the initial node. This means the same edges and nodes may appear multiple times in a path. 
    
    A path's \deff{length} is defined as the number of edges. 
    
    We say that a path is \deff{locally consistent}, iff any two subsequent edges are compatible. 
    
    A path is called \deff{globally consistent}, iff any two edges of the path are consistent. 
    
    An edge of a path is called \deff{globally consistent}, iff it is consistent with all the preceding edges of the path.
\end{defn}

Notice that two edges $e, e'$ going from $p, p'$, respectively, are consistent iff the set $\{(p, l_{F}(e)), (p', l_{F}(e'))\}$ is a partial one-to-one mapping from $P_{n}$ to $H_{n}$. 

Figure \ref{fig1} illustrates a simple example of $\mathcal{G}_{F}$ for $n = 3$.

\begin{figure}[h!t]
\centering
\begin{tikzpicture}[->, >=stealth, semithick, node distance=4cm, initial text=" "]
    \node[state] (q0) {0};
    \node[state, right of=q0] (q1) {1};
    \node[state, below of=q1] (q2) {2};
    \node[state, below of=q0] (q3) {3};

    \draw   (q0) edge[above, bend left=15] node{0} (q1)
            (q0) edge[above, bend right=15] node{1} (q1)
            (q0) edge[left] node{2} (q3)
            (q1) edge[right, bend left=15] node{0} (q2)
            (q1) edge[right, bend right=15] node{1} (q2)
            (q1) edge[above, bend right=15] node{2} (q3)
            (q2) edge[loop right] node{0} (q2)
            (q2) edge[loop below] node{1} (q2)
            (q2) edge[above, bend right=15] node{2} (q3)
            (q3) edge[above, bend right=15] node{0} (q1)
            (q3) edge[above, bend right=15] node{1} (q2)
            (q3) edge[loop left] node{2} (q3);
\end{tikzpicture}
\caption{- example of $\mathcal{G}_{F}$}
\label{fig1}
\end{figure}

Fixing a strategy $(s, F)$ of \P and assuming \D never gives two contradictory answers in a row, we can identify a play in the game with a locally consistent path in $\mathcal{G}_{F}$ of length $s$. A play results in \DD's victory, iff the last edge of the mentioned path is globally consistent. Thus, there is a winning anti-strategy for \D against $(s, F)$, iff there is a locally consistent path in $\mathcal{G}_{F}$ of length $s$ with the last edge being globally consistent. In particular, there is a winning canonical anti-strategy for \D against $(s, F)$, iff there is a globally consistent path in $\mathcal{G}_{F}$ of length $s$.

We can now formulate the main theorem of this section in the following way.

\begin{thm}
\label{main}
    Let $n \geq 3$ and $(s, F)$ be a strategy for \PP. Then, there is a locally consistent path in $\mathcal{G}_{F}$ of length $s$ with the last edge being globally consistent.
\end{thm}
\begin{proof}
    By Proposition \ref{canonical anti-strategy}, it is enough to deal with the case $s \geq n + 1$. 
    
    We begin by constructing a globally consistent path in $\mathcal{G}_{F}$ of length $n$. This corresponds to \D playing according to a canonical anti-strategy during the first $n$ moves. 
    
    As was shown in Proposition \ref{loop}, if the path we have constructed so far passes through the same node at least twice, then \D wins against any $s$ just by using this exact canonical anti-strategy for the rest of the game. Thus, we may assume each node appears exactly once in the path. So, after the path is constructed, all the elements of $H_{n}$ viewed as labels of the path's edges are exhausted along the way. 
    
    Without loss of generality, we assume the exact path we get is as shown in Figure \ref{fig2}.
    
\begin{figure}[ht]
\centering
\begin{tikzpicture}[->, >=stealth, semithick, node distance=2.4cm, initial text="", initial/.style={draw=none}]
    \node[state] (qn) {$n$};
    \node[state, right of=qn] (qn1) {$n-$1};
    \node[state, initial, right of=qn1] (qnull) {...};
    \node[state, right of=qnull] (q3) {3};
    \node[state, right of=q3] (q2) {2};
    \node[state, right of=q2] (q1) {1};
    \node[state, right of=q1] (q0) {0};

    \draw   (qn) edge[above] node{$n-1$} (qn1)
            (qn1) edge[above] node{$n-2$} (qnull)
            (qnull) edge[above] node{3} (q3)
            (q3) edge[above] node{2} (q2)
            (q2) edge[above] node{1} (q1)
            (q1) edge[above] node{0} (q0);
\end{tikzpicture}
\caption{- globally consistent path in $\mathcal{G}_{F}$ of length $n$}
\label{fig2}
\end{figure}

    Starting from this configuration, our goal is to gradually build up more and more information about $\mathcal{G}_{F}$ by subsequently eliminating cases against which we can build winning anti-strategies for \D (i.e. cases for which we can find a locally consistent path of the prescribed length with the last edge globally consistent).
    
    The final goal is to reduce the whole analysis to the configuration shown in Figure \ref{fig3}.

\begin{figure}[ht]
\centering
\begin{tikzpicture}[->, >=stealth, semithick, node distance=2.4cm, initial text="", initial/.style={draw=none}]
    \node[state, initial, right of=qn1] (qnull) {...};
    \node[state, right of=qnull] (q3) {3};
    \node[state, right of=q3] (q2) {2};
    \node[state, right of=q2] (q1) {1};
    \node[state, right of=q1] (q0) {0};

    \draw   (qnull) edge[above] node{3} (q3)
            (q3) edge[above, bend left=35] node{2} (q2)
            (q3) edge[above] node{1} (q2)
            (q3) edge[above, bend right=35] node{0} (q2)
            (q2) edge[above, bend left=35] node{2} (q1)
            (q2) edge[above] node{1} (q1)
            (q2) edge[above, bend right=35] node{0} (q1)
            (q1) edge[above, bend left=35] node{2} (q0)
            (q1) edge[above] node{1} (q0)
            (q1) edge[above, bend right=35] node{0} (q0);
\end{tikzpicture}
\caption{}
\label{fig3}
\end{figure}

    We first show that for any $\mathcal{G}_F$ containing labeled subgraph as in Figure \ref{fig3} and any $s \geq n + 1$, we can find a locally consistent path of length $s$ with the last edge globally consistent.

    \begin{lemma}
    \label{lmmain}
        Assume $\mathcal{G}_{F}$ contains labeled subgraph as in Figure \ref{fig3}. Then, given $s \geq n + 1$, there is a locally consistent path of length $s$ in $\mathcal{G}_F$ with the last edge globally consistent.
    \end{lemma}
    \begin{proof}
    
        To find such a path, we analyze four different cases depending on where the edge from $0 \in P_{n}$ and labeled by $0 \in H_{n}$ goes. This is equivalent to analyzing the value $a$ so that $F(0, 0) = a$.
        \begin{enumerate}
            \item $a = 0$. The desired path is shown in Figure \ref{fig4}. (We actually depict a subgraph. For each $s$, one gets the desired path just by following the subgraph for $s$ steps.) 

\begin{figure}[ht]
\centering
\begin{tikzpicture}[->, >=stealth, semithick, node distance=2.4cm, initial text="", initial/.style={draw=none}]
    \node[state, initial, right of=qn1] (qnull) {...};
    \node[state, right of=qnull] (q3) {3};
    \node[state, right of=q3] (q2) {2};
    \node[state, right of=q2] (q1) {1};
    \node[state, right of=q1] (q0) {0};

    \draw   (qnull) edge[above] node{3} (q3)
            (q3) edge[above] node{2} (q2)
            (q2) edge[above] node{1} (q1)
            (q1) edge[above, color=red, text=black, dashed] node{2} (q0)
            (q0) edge[loop right] node{0} (q0);
\end{tikzpicture}
\caption{}
\label{fig4}
\end{figure}

            Given any $s \geq n + 1$ and following such a path, we are guaranteed that the final question is $0 \in P_{n}$ with the answer $0 \in H_{n}$. This is a globally consistent answer. 
            
            In the above picture, only one edge is not globally consistent. We will color such edges red and mark them dashed to distinguish parts of the path ``undesirable" for \D (i.e., following such path, \D loses the game if the last answer corresponds to the red dashed edge in the picture).
            
            \item $a = 1$. We have two different subgraphs shown in Figure \ref{fig5}. While each is not enough by itself, for any $s \geq n + 1$, one of the graphs unrolls into a desired path.

            We say that such subgraphs form a \deff{cover-by-two}, i.e., given $s \geq n + 1$, following such subgraphs for $s$-many steps results in two locally consistent paths with at least one of them having last edge globally consistent. In a sense, every $s$ is ``covered" by one of the two paths.

            In the situation depicted in Figure \ref{fig5}, for $s - n - 1$ even, the upper graph provides the desired path. Otherwise, the bottom one is correct.
            
\begin{figure}[ht]
\centering
\begin{tikzpicture}[->, >=stealth, semithick, node distance=2.4cm, initial text="", initial/.style={draw=none}]
    \node[state, initial] (qnull) {...};
    \node[state, right of=qnull] (q3) {3};
    \node[state, right of=q3] (q2) {2};
    \node[state, right of=q2] (q1) {1};
    \node[state, right of=q1] (q0) {0};
    \node[state, initial, below of=qnull] (qnullP) {...};
    \node[state, right of=qnullP] (q3P) {3};
    \node[state, right of=q3P] (q2P) {2};
    \node[state, right of=q2P] (q1P) {1};
    \node[state, right of=q1P] (q0P) {0};

    \draw   (qnull) edge[above] node{3} (q3)
            (q3) edge[above] node{2} (q2)
            (q2) edge[above] node{0} (q1)
            (q1) edge[above] node{1} (q0)
            (q0) edge[above, bend right=35, color=red, text=black, dashed] node{0} (q1)
            (qnullP) edge[above] node{3} (q3P)
            (q3P) edge[above] node{2} (q2P)
            (q2P) edge[above] node{1} (q1P)
            (q1P) edge[above, color=red, text=black, dashed] node{2} (q0P)
            (q0P) edge[above, bend right=35] node{0} (q1P);
\end{tikzpicture}
\caption{- cover-by-two for case $a = 1$}
\label{fig5}
\end{figure}

        \item $a = 2$. Similarly to the previous case, we have two graphs shown in Figure \ref{fig6} forming a cover-by-two. For $s - n - 2$ congruent to 1 or 2 mod 3, the top graph gives the desired path. For $s - n - 2$ congruent to 1 or 0 mod 3, the bottom one is correct (notice an overlap for $s - n - 2$ congruent to 1 mod 3).

\begin{figure}[ht]
\centering
\begin{tikzpicture}[->, >=stealth, semithick, node distance=2.4cm, initial text="", initial/.style={draw=none}]
    \node[state, initial] (qnull) {...};
    \node[state, right of=qnull] (q3) {3};
    \node[state, right of=q3] (q2) {2};
    \node[state, right of=q2] (q1) {1};
    \node[state, right of=q1] (q0) {0};
    \node[state, initial, below of=qnull] (qnullP) {...};
    \node[state, right of=qnullP] (q3P) {3};
    \node[state, right of=q3P] (q2P) {2};
    \node[state, right of=q2P] (q1P) {1};
    \node[state, right of=q1P] (q0P) {0};

    \draw   (qnull) edge[above] node{3} (q3)
            (q3) edge[above] node{0} (q2)
            (q2) edge[above] node{1} (q1)
            (q1) edge[above] node{2} (q0)
            (q0) edge[above, bend right=35, color=red, text=black, dashed] node{0} (q2)
            (qnullP) edge[above] node{3} (q3P)
            (q3P) edge[above] node{2} (q2P)
            (q2P) edge[above] node{1} (q1P)
            (q1P) edge[above, color=red, text=black, dashed] node{2} (q0P)
            (q0P) edge[above, bend right=35] node{0} (q2P);
\end{tikzpicture}
\caption{- cover-by-two for case $a = 2$}
\label{fig6}
\end{figure}

        \item $a \geq 3$. Again, we have two graphs shown in Figure \ref{fig7} forming a cover-by-two. This time, it is harder (although possible) to specify which graph unrolls into the desired path depending on the value of $s$.

\begin{figure}[ht]
\centering
\begin{tikzpicture}[->, >=stealth, semithick, node distance=2.4cm, initial text="", initial/.style={draw=none}]
    \node[state, initial] (qnull) {...};
    \node[state, right of=qnull] (q3) {3};
    \node[state, right of=q3] (q2) {2};
    \node[state, right of=q2] (q1) {1};
    \node[state, right of=q1] (q0) {0};
    \node[state, initial, below=2cm of qnull] (qnullP) {...};
    \node[state, right of=qnullP] (q3P) {3};
    \node[state, right of=q3P] (q2P) {2};
    \node[state, right of=q2P] (q1P) {1};
    \node[state, right of=q1P] (q0P) {0};

    \draw   (qnull) edge[above] node{3} (q3)
            (q3) edge[above, color=red, text=black, dashed] node{2} (q2)
            (q2) edge[above] node{1} (q1)
            (q1) edge[above, color=red, text=black, dashed] node{2} (q0)
            (q0) edge[above, bend right=30] node{0} (qnull)
            (qnullP) edge[above] node{3} (q3P)
            (q3P) edge[above] node{2} (q2P)
            (q2P) edge[above, color=red, text=black, dashed] node{0} (q1P)
            (q1P) edge[above] node{1} (q0P)
            (q0P) edge[above, bend right=30, color=red, text=black, dashed] node{0} (qnullP);
\end{tikzpicture}
\caption{- cover-by-two for case $a \geq 3$}
\label{fig7}
\end{figure}
            
        \end{enumerate}
        This exhausts all possible cases.
    \end{proof}

    \hfill

    The rest of the proof goes by subsequently achieving certain \textit{goals} by means of \textit{derivations}. All this is just a case analysis presented in a way that is hopefully easy to follow. All our derivations are of the form $F(a, b) = c$, which means that, assuming $F(a, b) = c'$ for $c' \neq c$, under the considered assumptions on the structure of the graph, the existence of a locally consistent path of the prescribed length with the last edge globally consistent is immediate, or is shown by a subsequent argument. The goals themselves are meant to represent environments for the arguments justifying the corresponding derivations. Whenever a goal is fulfilled, the assumption of the corresponding derivation is added to the overall list of assumptions on the current graph structure.
    
    \goal{Goal I}: derive $F(2, 0) = 1$. (Graph-theoretically, we want to show that the edge from $2 \in P_{n}$ labeled by $0 \in H_{n}$ goes to $1 \in P_{n}$; we omit using this longer expression from now on.)
    
    Let $F(2, 0) = a$, with $a \neq 1$. We immediately derive $a = 0$, and then derive $F(0, 1) = 1$, as is shown in Figure \ref{fig8}.

\begin{figure}[ht]
\centering
\begin{tikzpicture}[->, >=stealth, semithick, node distance=2.4cm, initial text="", initial/.style={draw=none}]
    \node[state, initial, right of=qn1] (qnull) {...};
    \node[state, right of=qnull] (q3) {3};
    \node[state, right of=q3] (q2) {2};
    \node[state, right of=q2] (q1) {1};
    \node[state, right of=q1] (q0) {0};

    \draw   (qnull) edge[above] node{3} (q3)
            (q3) edge[above] node{2} (q2)
            (q2) edge[above] node{1} (q1)
            (q1) edge[above] node{0} (q0)
            (q0) edge[above, bend right=60] node{1 \circled[0.5]{$\ast$}} (q1)
            (q2) edge[below, bend right=35] node{0} (q0);
\end{tikzpicture}
\caption{}
\label{fig8}
\end{figure}

     We mark the edges we can derive immediately by \circled[0.5]{$\ast$} from now on. Assume now that $F(3, 1) = 1$. This results in a cover-by-two, as is shown in Figure \ref{fig9}.

\begin{figure}[ht]
\centering
\begin{tikzpicture}[->, >=stealth, semithick, node distance=2.4cm, initial text="", initial/.style={draw=none}]
    \node[state, initial] (qnull) {...};
    \node[state, right of=qnull] (q3) {3};
    \node[state, right of=q3] (q2) {2};
    \node[state, right of=q2] (q1) {1};
    \node[state, right of=q1] (q0) {0};
    \node[state, initial, below=2cm of qnull] (qnullP) {...};
    \node[state, right of=qnullP] (q3P) {3};
    \node[state, right of=q3P] (q1P) {1};
    \node[state, right of=q1P] (q0P) {0};

    \draw   (qnull) edge[above] node{3} (q3)
            (q3) edge[above] node{2} (q2)
            (q2) edge[above] node{1} (q1)
            (q1) edge[above] node{0} (q0)
            (q0) edge[above, bend right=35, color=red, text=black, dashed] node{1} (q1)
            (qnullP) edge[above] node{3} (q3P)
            (q3P) edge[above] node{1} (q1P)
            (q1P) edge[above] node{0} (q0P)
            (q0P) edge[above, bend right=35, color=red, text=black, dashed] node{1} (q1P);
\end{tikzpicture}
\caption{- cover-by-two derived under the assumption $F(3, 1) = 1$}
\label{fig9}
\end{figure}

    So let $F(3, 1) = b$, with $b \neq 1$. We may further derive that either $b = 0$ or $b = 2$.

    Assume $b = 0$. The resulting graph is shown in Figure \ref{fig10}.

\begin{figure}[h!t]
\centering
\begin{tikzpicture}[->, >=stealth, semithick, node distance=2.4cm, initial text="", initial/.style={draw=none}]
    \node[state, initial, right of=qn1] (qnull) {...};
    \node[state, right of=qnull] (q3) {3};
    \node[state, right of=q3] (q2) {2};
    \node[state, right of=q2] (q1) {1};
    \node[state, right of=q1] (q0) {0};

    \draw   (qnull) edge[above] node{3} (q3)
            (q3) edge[above] node{2} (q2)
            (q2) edge[above] node{1} (q1)
            (q1) edge[above] node{0} (q0)
            (q0) edge[above, bend right=60] node{1} (q1)
            (q2) edge[below, bend right=35] node{0} (q0)
            (q3) edge[above, bend right=60] node{1} (q0);
\end{tikzpicture}
\caption{}
\label{fig10}
\end{figure}

    By analyzing all the possible values for $c = F(0, 2)$, we see that any such choice immediately yields the desired path for any $s$.

    Thus, we derive $b = 2$. Figure \ref{fig11} shows this graph together with an edge that can be immediately derived, as well.

\begin{figure}[ht]
\centering
\begin{tikzpicture}[->, >=stealth, semithick, node distance=2.4cm, initial text="", initial/.style={draw=none}]
    \node[state, initial, right of=qn1] (qnull) {...};
    \node[state, right of=qnull] (q3) {3};
    \node[state, right of=q3] (q2) {2};
    \node[state, right of=q2] (q1) {1};
    \node[state, right of=q1] (q0) {0};

    \draw   (qnull) edge[above] node{3} (q3)
            (q3) edge[above, bend left=35] node{2} (q2)
            (q3) edge[above, bend right=35] node{1} (q2)
            (q2) edge[above] node{1} (q1)
            (q2) edge[above, bend right=60] node{0} (q0)
            (q1) edge[above, bend right=35] node{0} (q0)
            (q0) edge[above, bend right=35] node{1} (q1)
            (q0) edge[above, bend right=75] node{2 \circled[0.5]{$\ast$}} (q1);
\end{tikzpicture}
\caption{}
\label{fig11}
\end{figure}

    Assuming $F(2, 2) = 0$, we derive a cover-by-two shown in Figure \ref{fig12}.

\begin{figure}[ht]
\centering
\begin{tikzpicture}[->, >=stealth, semithick, node distance=2.4cm, initial text="", initial/.style={draw=none}]
    \node[state, initial] (qnull) {...};
    \node[state, right of=qnull] (q3) {3};
    \node[state, right of=q3] (q2) {2};
    \node[state, right of=q2] (q0) {0};
    \node[state, right of=q0] (q1) {1};
    \node[state, initial, below of=qnull] (qnullP) {...};
    \node[state, right of=qnullP] (q3P) {3};
    \node[state, right of=q3P] (q2P) {2};
    \node[state, right of=q2P] (q0P) {0};
    \node[state, right of=q0P] (q1P) {1};

    \draw   (qnull) edge[above] node{3} (q3)
            (q3) edge[above] node{1} (q2)
            (q2) edge[above] node{0} (q0)
            (q0) edge[above] node{2} (q1)
            (q1) edge[above, bend right=35, color=red, text=black, dashed] node{0} (q0)
            (qnullP) edge[above] node{3} (q3P)
            (q3P) edge[above] node{1} (q2P)
            (q2P) edge[above] node{2} (q0P)
            (q0P) edge[above, color=red, text=black, dashed] node{1} (q1P)
            (q1P) edge[above, bend right=35] node{0} (q0P);
\end{tikzpicture}
\caption{- cover-by-two derived under the assumption $F(2, 2) = 0$}
\label{fig12}
\end{figure}

    Thus, we derive $F(2, 2) = 1$. The resulting graph is shown in Figure \ref{fig13}.

\begin{figure}[h!t]
\centering
\begin{tikzpicture}[->, >=stealth, semithick, node distance=2.4cm, initial text="", initial/.style={draw=none}]
    \node[state, initial, right of=qn1] (qnull) {...};
    \node[state, right of=qnull] (q3) {3};
    \node[state, right of=q3] (q2) {2};
    \node[state, right of=q2] (q1) {1};
    \node[state, right of=q1] (q0) {0};

    \draw   (qnull) edge[above] node{3} (q3)
            (q3) edge[above, bend left=35] node{2} (q2)
            (q3) edge[above, bend right=35] node{1} (q2)
            (q2) edge[above, bend right=35] node{1} (q1)
            (q2) edge[above, bend left=35] node{2} (q1)
            (q2) edge[above, bend right=60] node{0} (q0)
            (q1) edge[above, bend right=35] node{0} (q0)
            (q0) edge[above, bend right=35] node{1} (q1)
            (q0) edge[above, bend right=75] node{2} (q1);
\end{tikzpicture}
\caption{}
\label{fig13}
\end{figure}
    
    Let $d = F(3,0)$. We immediately derive that, either $d = 0$, or $d = 1$, or $d = 2$. 
    
    The case $d = 0$ yields a cover-by-two shown in Figure \ref{fig14}.

\begin{figure}[ht]
\centering
\begin{tikzpicture}[->, >=stealth, semithick, node distance=2.4cm, initial text="", initial/.style={draw=none}]
    \node[state, initial] (qnull) {...};
    \node[state, right of=qnull] (q3) {3};
    \node[state, right of=q3] (q2) {2};
    \node[state, right of=q2] (q0) {0};
    \node[state, right of=q0] (q1) {1};
    \node[state, initial, below=2cm of qnull] (qnullP) {...};
    \node[state, right of=qnullP] (q3P) {3};
    \node[state, right of=q3P] (q0P) {0};
    \node[state, right of=q0P] (q1P) {1};

    \draw   (qnull) edge[above] node{3} (q3)
            (q3) edge[above] node{2} (q2)
            (q2) edge[above] node{0} (q0)
            (q0) edge[above] node{1} (q1)
            (q1) edge[above, bend right=35, color=red, text=black, dashed] node{0} (q0)
            (qnullP) edge[above] node{3} (q3P)
            (q3P) edge[above] node{0} (q0P)
            (q0P) edge[above] node{1} (q1P)
            (q1P) edge[above, bend right=35, color=red, text=black, dashed] node{0} (q0P);
\end{tikzpicture}
\caption{- cover-by-two derived under the assumption $d = 0$}
\label{fig14}
\end{figure}

    The case $d = 2$ yields a cover-by-two shown in Figure \ref{fig15}.

\begin{figure}[ht]
\centering
\begin{tikzpicture}[->, >=stealth, semithick, node distance=2.4cm, initial text="", initial/.style={draw=none}]
    \node[state, initial] (qnull) {...};
    \node[state, right of=qnull] (q3) {3};
    \node[state, right of=q3] (q2) {2};
    \node[state, right of=q2] (q1) {1};
    \node[state, right of=q1] (q0) {0};
    \node[state, initial, below of=qnull] (qnullP) {...};
    \node[state, right of=qnullP] (q3P) {3};
    \node[state, right of=q3P] (q2P) {2};
    \node[state, right of=q2P] (q1P) {1};
    \node[state, right of=q1P] (q0P) {0};

    \draw   (qnull) edge[above] node{3} (q3)
            (q3) edge[above] node{1} (q2)
            (q2) edge[above] node{2} (q1)
            (q1) edge[above] node{0} (q0)
            (q0) edge[above, bend right=35, color=red, text=black, dashed] node{1} (q1)
            (qnullP) edge[above] node{3} (q3P)
            (q3P) edge[above] node{0} (q2P)
            (q2P) edge[above] node{2} (q1P)
            (q1P) edge[above, color=red, text=black, dashed] node{0} (q0P)
            (q0P) edge[above, bend right=35] node{1} (q1P);
\end{tikzpicture}
\caption{- cover-by-two derived under the assumption $d = 2$}
\label{fig15}
\end{figure}

    Lastly, assume $d = 1$. This results in the graph depicted in Figure \ref{fig16}, from which one can easily find a desired path for any $s$.

\begin{figure}[h!t]
\centering
\begin{tikzpicture}[->, >=stealth, semithick, node distance=2.4cm, initial text="", initial/.style={draw=none}]
    \node[state, initial, right of=qn1] (qnull) {...};
    \node[state, right of=qnull] (q3) {3};
    \node[state, right of=q3] (q2) {2};
    \node[state, right of=q2] (q1) {1};
    \node[state, right of=q1] (q0) {0};

    \draw   (qnull) edge[above] node{3} (q3)
            (q3) edge[above, bend left=35] node{2} (q2)
            (q3) edge[above, bend right=35] node{1} (q2)
            (q3) edge[above, bend right=60] node{0} (q1)
            (q2) edge[above, bend right=35] node{1} (q1)
            (q2) edge[above, bend left=35] node{2} (q1)
            (q2) edge[above, bend right=60] node{0} (q0)
            (q1) edge[above, bend right=35] node{0} (q0)
            (q1) edge[above] node{1 \circled[0.5]{$\ast$}} (q0)
            (q0) edge[above, bend right=45] node{1} (q1)
            (q0) edge[above, bend right=90] node{2} (q1);
\end{tikzpicture}
\caption{- an additional derivation for case $d = 1$}
\label{fig16}
\end{figure}

    Altogether, this fulfills the first goal. In other words, we derive $F(2, 0) = 1$, as is shown in Figure \ref{fig17}.

\begin{figure}[h!t]
\centering
\begin{tikzpicture}[->, >=stealth, semithick, node distance=2.4cm, initial text="", initial/.style={draw=none}]
    \node[state, initial, right of=qn1] (qnull) {...};
    \node[state, right of=qnull] (q3) {3};
    \node[state, right of=q3] (q2) {2};
    \node[state, right of=q2] (q1) {1};
    \node[state, right of=q1] (q0) {0};

    \draw   (qnull) edge[above] node{3} (q3)
            (q3) edge[above] node{2} (q2)
            (q2) edge[above, bend right=35] node{0} (q1)
            (q2) edge[above, bend left=35] node{1} (q1)
            (q1) edge[above, bend right=35] node{0} (q0)
            (q1) edge[above, bend left=35] node{1 \circled[0.5]{$\ast$}} (q0);
\end{tikzpicture}
\caption{}
\label{fig17}
\end{figure}

    \goal{Goal II}: derive $F(3, 1) = 2$. Assuming the derivation is valid, we can repeat the argument to get $F(3, 0) = 2$, just by interchanging edges labeled by $1$ with edges labeled by $0$, as the graph depicted in Figure \ref{fig17} remains the same.

    Let $F(3, 1) = a$. We immediately derive that, either $a = 0$, or $a = 1$, or $a = 2$.

    Assume $a = 0$. We get the graph depicted in Figure \ref{fig18}.

\begin{figure}[h!t]
\centering
\begin{tikzpicture}[->, >=stealth, semithick, node distance=2.4cm, initial text="", initial/.style={draw=none}]
    \node[state, initial, right of=qn1] (qnull) {...};
    \node[state, right of=qnull] (q3) {3};
    \node[state, right of=q3] (q2) {2};
    \node[state, right of=q2] (q1) {1};
    \node[state, right of=q1] (q0) {0};

    \draw   (qnull) edge[above] node{3} (q3)
            (q3) edge[above] node{2} (q2)
            (q3) edge[above, bend right=60] node{1} (q0)
            (q2) edge[above, bend right=35] node{0} (q1)
            (q2) edge[above, bend left=35] node{1} (q1)
            (q1) edge[above, bend right=35] node{0} (q0)
            (q1) edge[above, bend left=35] node{1} (q0)
            (q0) edge[above, bend right=60] node{2 \circled[0.5]{$\ast$}} (q2);
\end{tikzpicture}
\caption{}
\label{fig18}
\end{figure}

    This leads to a cover-by-two shown in Figure \ref{fig19}.

\begin{figure}[ht]
\centering
\begin{tikzpicture}[->, >=stealth, semithick, node distance=2.4cm, initial text="", initial/.style={draw=none}]
    \node[state, initial] (qnull) {...};
    \node[state, right of=qnull] (q3) {3};
    \node[state, right of=q3] (q0) {0};
    \node[state, right of=q0] (q2) {2};
    \node[state, right of=q2] (q1) {1};
    \node[state, initial, below of=qnull] (qnullP) {...};
    \node[state, right of=qnullP] (q3P) {3};
    \node[state, right of=q3P] (q0P) {0};
    \node[state, right of=q0P] (q2P) {2};
    \node[state, right of=q2P] (q1P) {1};

    \draw   (qnull) edge[above] node{3} (q3)
            (q3) edge[above] node{1} (q0)
            (q0) edge[above] node{2} (q2)
            (q2) edge[above] node{0} (q1)
            (q1) edge[above, bend right=35, color=red, text=black, dashed] node{1} (q0)
            (qnullP) edge[above] node{3} (q3P)
            (q3P) edge[above] node{1} (q0P)
            (q0P) edge[above] node{2} (q2P)
            (q2P) edge[above, color=red, text=black, dashed] node{1} (q1P)
            (q1P) edge[above, bend right=35] node{0} (q0P);
\end{tikzpicture}
\caption{- cover-by-two derived under the assumption $a = 0$}
\label{fig19}
\end{figure}

    Assume $a = 1$. This is shown in Figure \ref{fig20}. We derive $F(1, 2) = 0$, and then, with this additional information, we derive $F(0, 0) = 2$. (The corresponding edge in the picture below is marked with \circled[0.5]{$\ast \ast$} to emphasize that its derivation comes second after the one marked with \circled[0.5]{$\ast$}.)

\begin{figure}[h!t]
\centering
\begin{tikzpicture}[->, >=stealth, semithick, node distance=2.4cm, initial text="", initial/.style={draw=none}]
    \node[state, initial, right of=qn1] (qnull) {...};
    \node[state, right of=qnull] (q3) {3};
    \node[state, right of=q3] (q2) {2};
    \node[state, right of=q2] (q1) {1};
    \node[state, right of=q1] (q0) {0};

    \draw   (qnull) edge[above] node{3} (q3)
            (q3) edge[above] node{2} (q2)
            (q3) edge[above, bend right=60] node{1} (q1)
            (q2) edge[above, bend right=35] node{0} (q1)
            (q2) edge[above, bend left=35] node{1} (q1)
            (q1) edge[above, bend right=35] node{0} (q0)
            (q1) edge[above, bend left=35] node{1} (q0)
            (q1) edge[above, bend left=75] node{2 \circled[0.5]{$\ast$}} (q0)
            (q0) edge[below, bend left=60] node{0 \circled[0.5]{$\ast \ast$}} (q2);
\end{tikzpicture}
\caption{}
\label{fig20}
\end{figure}

    This leads to a cover-by-two shown in Figure \ref{fig21}.

\begin{figure}[ht]
\centering
\begin{tikzpicture}[->, >=stealth, semithick, node distance=2.4cm, initial text="", initial/.style={draw=none}]
    \node[state, initial] (qnull) {...};
    \node[state, right of=qnull] (q3) {3};
    \node[state, right of=q3] (q1) {1};
    \node[state, right of=q1] (q0) {0};
    \node[state, right of=q0] (q2) {2};
    \node[state, initial, below of=qnull] (qnullP) {...};
    \node[state, right of=qnullP] (q3P) {3};
    \node[state, right of=q3P] (q2P) {2};
    \node[state, right of=q2P] (q1P) {1};
    \node[state, right of=q1P] (q0P) {0};

    \draw   (qnull) edge[above] node{3} (q3)
            (q3) edge[above] node{1} (q1)
            (q1) edge[above] node{2} (q0)
            (q0) edge[above] node{0} (q2)
            (q2) edge[above, bend right=35, color=red, text=black, dashed] node{1} (q1)
            (qnullP) edge[above] node{3} (q3P)
            (q3P) edge[above] node{2} (q2P)
            (q2P) edge[above] node{1} (q1P)
            (q1P) edge[above, color=red, text=black, dashed] node{2} (q0P)
            (q0P) edge[above, bend right=35] node{0} (q2P);
\end{tikzpicture}
\caption{- cover-by-two derived under the assumption $a = 1$}
\label{fig21}
\end{figure}

    Thus, the second goal is fulfilled. This leads to Figure \ref{fig22}. (Again, we mark by \circled[0.5]{$\ast$} the additional edge, which can be immediately derived.)

\begin{figure}[h!t]
\centering
\begin{tikzpicture}[->, >=stealth, semithick, node distance=2.4cm, initial text="", initial/.style={draw=none}]
    \node[state, initial, right of=qn1] (qnull) {...};
    \node[state, right of=qnull] (q3) {3};
    \node[state, right of=q3] (q2) {2};
    \node[state, right of=q2] (q1) {1};
    \node[state, right of=q1] (q0) {0};

    \draw   (qnull) edge[above] node{3} (q3)
            (q3) edge[above, bend left=35] node{2} (q2)
            (q3) edge[above] node{1} (q2)
            (q3) edge[above, bend right=35] node{0} (q2)
            (q2) edge[above, bend left=35] node{1} (q1)
            (q2) edge[above, bend right=35] node{0} (q1)
            (q1) edge[above, bend left=35] node{2 \circled[0.5]{$\ast$}} (q0)
            (q1) edge[above] node{1} (q0)
            (q1) edge[above, bend right=35] node{0} (q0);
\end{tikzpicture}
\caption{}
\label{fig22}
\end{figure}

    We derive that, either $F(2, 2) = 1$, or $F(2, 2) = 0$. 
    
    Assume $F(2, 2) = 0$. This is depicted in Figure \ref{fig23}.

\begin{figure}[h!t]
\centering
\begin{tikzpicture}[->, >=stealth, semithick, node distance=2.4cm, initial text="", initial/.style={draw=none}]
    \node[state, initial, right of=qn1] (qnull) {...};
    \node[state, right of=qnull] (q3) {3};
    \node[state, right of=q3] (q2) {2};
    \node[state, right of=q2] (q1) {1};
    \node[state, right of=q1] (q0) {0};

    \draw   (qnull) edge[above] node{3} (q3)
            (q3) edge[above, bend left=35] node{2} (q2)
            (q3) edge[above] node{1} (q2)
            (q3) edge[above, bend right=35] node{0} (q2)
            (q2) edge[above, bend left=35] node{1} (q1)
            (q2) edge[above, bend right=35] node{0} (q1)
            (q2) edge[above, bend right=60] node{2} (q0)
            (q1) edge[above, bend left=35] node{2} (q0)
            (q1) edge[above] node{1} (q0)
            (q1) edge[above, bend right=35] node{0} (q0)
            (q0) edge[above, bend right=75] node{0 \circled[0.5]{$\ast$}} (q1);
\end{tikzpicture}
\caption{}
\label{fig23}
\end{figure}

    This leads to a cover-by-two shown in Figure \ref{fig24}.

\begin{figure}[h!t]
\centering
\begin{tikzpicture}[->, >=stealth, semithick, node distance=2.4cm, initial text="", initial/.style={draw=none}]
    \node[state, initial] (qnull) {...};
    \node[state, right of=qnull] (q3) {3};
    \node[state, right of=q3] (q2) {2};
    \node[state, right of=q2] (q0) {0};
    \node[state, right of=q0] (q1) {1};
    \node[state, initial, below of=qnull] (qnullP) {...};
    \node[state, right of=qnullP] (q3P) {3};
    \node[state, right of=q3P] (q2P) {2};
    \node[state, right of=q2P] (q0P) {0};
    \node[state, right of=q0P] (q1P) {1};

    \draw   (qnull) edge[above] node{3} (q3)
            (q3) edge[above] node{1} (q2)
            (q2) edge[above] node{2} (q0)
            (q0) edge[above] node{0} (q1)
            (q1) edge[above, bend right=35, color=red, text=black, dashed] node{1} (q0)
            (qnullP) edge[above] node{3} (q3P)
            (q3P) edge[above] node{0} (q2P)
            (q2P) edge[above] node{2} (q0P)
            (q0P) edge[above, color=red, text=black, dashed] node{0} (q1P)
            (q1P) edge[above, bend right=35] node{1} (q0P);
\end{tikzpicture}
\caption{- cover-by-two for case $F(2, 2) = 0$}
\label{fig24}
\end{figure}

    Thus, we derive $F(2, 2) = 1$, depicted in Figure \ref{fig25}.

\begin{figure}[h!t]
\centering
\begin{tikzpicture}[->, >=stealth, semithick, node distance=2.4cm, initial text="", initial/.style={draw=none}]
    \node[state, initial, right of=qn1] (qnull) {...};
    \node[state, right of=qnull] (q3) {3};
    \node[state, right of=q3] (q2) {2};
    \node[state, right of=q2] (q1) {1};
    \node[state, right of=q1] (q0) {0};

    \draw   (qnull) edge[above] node{3} (q3)
            (q3) edge[above, bend left=35] node{2} (q2)
            (q3) edge[above] node{1} (q2)
            (q3) edge[above, bend right=35] node{0} (q2)
            (q2) edge[above, bend left=35] node{2} (q1)
            (q2) edge[above] node{1} (q1)
            (q2) edge[above, bend right=35] node{0} (q1)
            (q1) edge[above, bend left=35] node{2} (q0)
            (q1) edge[above] node{1} (q0)
            (q1) edge[above, bend right=35] node{0} (q0);
\end{tikzpicture}
\caption{- configuration for which Lemma \ref{lmmain} applies}
\label{fig25}
\end{figure}

    Applying the Lemma from the beginning, we finish the whole proof.
\end{proof}

Below we present an alternative way of proving a result similar to Theorem \ref{main}. Instead of simultaneously proving the statement for all $n \geq 3$, we proceed by induction. 

The first drawback is that instead of finding anti-strategies against all $s$, we can construct them for all $s$ above some threshold $s_{0}$ which depends on $n$. 

The second drawback is that we have to deal with the cases $n = 3$ and $n = 4$ separately. 

Since we are not aware of any way of dealing with those cases in a principally simpler way than just proving \ref{main}, we only show how one deals with the induction step. 

Even though the mentioned proof strategy gives a weaker result than the one already presented, we feel the overall analysis is more systematic and provides more intuition for dealing with the game $\mathcal{G}_{2}$.

We start by introducing a new notion which is used throughout the proof.

\begin{defn}
    Let $P_{n} = \{0, ..., n\}$ and $H_{n} = \{0, ..., n-1\}$ as before. We define a \deff{php-tree} for $n$ as a rooted labeled tree $\mathcal{T} = (V, E, l_{V}, l_{E})$ satisfying the following conditions
    \begin{itemize}
        \item every node $v \in V$ is labeled by an element $l_{V}(v)$ from $P_{n}$;
        \item every edge $e \in E$ is labeled by an element $l_{E}(e)$ from $H_{n}$;
        \item the \textit{branching} number of any node on the $k$th \textit{level} is $\leq n - k$, where a node $v$ is said to \textit{belong to the $k$th level} iff the distance from $v$ to the root is $k$ and the branching number of a node $v$ on the $k$th level is defined as the number of edges connecting $v$ to its' children, i.e. nodes on the $(k + 1)$st level;
        \item along any path starting from the root, any label of a node appears at most once;
        \item along any path starting from the root, any label of an edge appears at most once.
    \end{itemize}
\end{defn}

In Figure \ref{php1}, we show a simple php-tree for $n = 3$.

\label{php tree illustration}
\begin{figure}[h!t]
\centering
\begin{tikzpicture}[-, >=stealth, semithick, node distance=2.4cm, initial text="", initial/.style={draw=none}]
    \node[state] (q0root) {0};
    \node[state, below left of=q0root] (q1top) {1};
    \node[state, below right of=q0root] (q3top) {3};
    \node[state, below left of=q1top] (q3mid) {3};
    \node [state, below right of=q1top] (q2mid) {2};
    \node [state, below of=q3mid] (q2bot) {2};
    \node [state, below right of=q3top] (q1mid) {1};
    \node [state, below of=q1mid] (q2botP) {2};

    \draw   (q0root) edge[above left] node{0} (q1top)
            (q0root) edge[above right] node{2} (q3top)
            (q1top) edge[above left] node{1} (q3mid)
            (q1top) edge[above right] node{2} (q2mid)
            (q3mid) edge[left] node{2} (q2bot)
            (q3top) edge[above right] node{0} (q1mid)
            (q1mid) edge[right] node{1} (q2botP);
\end{tikzpicture}
\caption{- example of a simple php-tree}
\label{php1}
\end{figure}

\begin{defn}
    Let $\mathcal{T}$ be a php-tree. The \deff{depth} of $\mathcal{T}$ is defined as the maximum level of a node of $\mathcal{T}$. 
    
    $\mathcal{T}$ is said to be \deff{complete}, iff its depth is $n$, and every node on the $k$th level has branching number exactly $n - k$. 
    
    $\mathcal{T}$ is said to be \deff{symmetric}, iff the label of a node $v$ and the label of an edge connecting $v$ to its child uniquely determines the label of a child, i.e. for any two nodes $v$ and $v'$ with the same label and any two edges $e, e'$ connecting $v$ to its child node $w$ and $v'$ to $w'$, respectively, if the labels of $e, e'$ are the same, then the labels of $w$ and $w'$ are the same, as well.
\end{defn}

Notice that the php-tree, as in the above picture, is not complete nor symmetric.

Symmetric php-trees naturally represent the structure of all plays against \PP's strategy, with \D utilizing canonical anti-strategies. This is formalized as follows. 

\begin{prop}
\label{php-tree canonical anti-strategy}
    Let $n \geq 3$ and $F$ be a functional part of \PP's strategy. Then, there is an associated symmetric php-tree $\mathcal{T}_{F}$ with the property that $\mathcal{T}_{F}$ is complete iff, for any $s \geq n + 1$, there is no winning canonical anti-strategy for \D to play against $(s, F)$. 
    
    This means that during each play against $(s, F)$, there will be a moment where \D cannot make her move compatibly with the canonical anti-strategy she has chosen.
\end{prop}
\begin{proof}
    We construct the tree as follows. 
    
    Label the root as $F(\emptyset)$. 
    
    Let $v_{i}$ be a node that we have already constructed together with its label $l_{V}(v_{i}) \in P_{n}$.
    
    Let $(v_{0},v_{1}, ..., v_{i-1}, v_{i})$ be the sequence of nodes appearing along the unique path starting from the root $v_{0}$ and ending in the node $v_{i}$. 
    
    For any $h \in H_{n}$ such that answering $h$ to a question $l_{V}(v_{i})$ results in a valid play according to a canonical anti-strategy for \D assuming \D has already answered $l_{E}(\{v_{j - 1}, v_{j}\})$ to $l_{V}(v_{j - 1})$ for all $0 < j \leq i$, we put a new node $w_{h}$ into the tree and connect it to $v_{i}$ via an edge $\{v_{i}, w_{h}\}$ iff $F(l_{V}(v_{i}), h)$ equals to $p \in P_{n}$ which has not yet appeared as a label of any node $v_{j}$ for $j \leq i$. We label the added $w_{h}$ as $F(l_{V}(v_{i}), h)$ and the edge $\{v_{i}, w_{h}\}$ as $h$.

    Notice that the tree we are constructing is, indeed, a php-tree. That the resulting labeled tree is symmetric follows from the fact that the label of a child node $w$ of $v$ depends only on the labels of $v$ and $\{v, w\}$, since $l_{V}(w) = F(l_{V}(v), l_{E}\{v, w\})$.

    Finally, the resulting php-tree is not complete iff there is a valid play against $(s, F)$ with \D using a canonical anti-strategy so that the records produced are $(p_{0}, h_{0}), ..., (p_{i}, h_{i})$ with $F(p_{i}, h_{i}) = p_{j}$ for $j \leq i$. As we have already seen in \ref{loop}, this implies that the chosen anti-strategy is a winning one for \DD. 
    
    The tree is complete iff in all games against $(s, F)$, with \D utilizing every possible canonical anti-strategy, at some point all $H_{n}$ are eventually exhausted, and so such plays cannot be continued using the rules of a canonical anti-strategy.
\end{proof}

\begin{thm}
\label{main1}
    Let $n \geq 5$. Assume \D has winning anti-strategies for $n - 1$ and $n - 2$ against all strategies $(s', F')$ for $n - 1$ and $(s'', F'')$ for $n - 2$, respectively, where $s', s'' \geq s_{0}$ for some particular threshold $s_{0}$. Then, \D has a winning anti-strategy against any $(s,F)$ for $s$ enjoying $s \geq \max\{s_{0} + 1, 2(n - 1) + 2\}$.
\end{thm}
\begin{proof}
    Let $F(\emptyset)$ be $n$. There are two principal ways how \D can utilize winning anti-strategies for $m < n$ to come up with a winning anti-strategy against $(s, F)$.
    \begin{itemize}
        \item \goal{Commit to the root}: \D chooses an answer $h$ for the first \PP's question $n \in P_{n}$ and then never answers $h$ to any subsequent \PP's questions. If she is guaranteed not to encounter the question $n$ anymore, then she can use a winning anti-strategy for $A'_{n - 1} = P_{n} \setminus \{n\}$ and $B'_{n - 1} = H_{n} \setminus \{h\}$ against $s' = s - 1$ and $F'$ which equals $F$ restricted to $A'_{n - 1} \times B'_{n - 1}$. 
        
        Formally speaking, the above assumption guarantees only that for any $(p', h') \in A'_{n - 1} \times B'_{n - 1})$ it holds that $F'(p', h') = F(p', h') \in B'_{n - 1}$ whenever the record $(p', h')$ \textit{can actually appear} in a play against $(s, F)$ starting with the record $(n, h)$. However, since we care only about \textit{admissible} records, we may assume $F'$ is properly defined on \textit{all} the records from $A'_{n - 1} \times B'_{n - 1}$ with values in $B'_{n - 1}$. 
        
        Finally, to properly utilize winning anti-strategy for $n - 1$ we need to relabel $B'_{n - 1}$ by elements of $H_{n - 1} = \{0, ..., n - 2\}$ via an arbitrary bijection.

        \item \goal{Forbid holes}: \D chooses different $h_{1}, ..., h_{i} \in H_{n}$ and plays without ever using those elements, i.e. she never answers $h_{j}$ for $j \leq i$ to any of the \PP's questions. If she is guaranteed not to encounter questions $p_{1}, ..., p_{i} \in P_{n}$ with $p_{j}$ all different, then she can use a winning anti-strategy for $A'_{n - i} = P_{n} \setminus \{p_{1}, ..., p_{i}\}$ and $B'_{n - i} = H_{n} \setminus \{h_{1}, ..., h_{i}\}$ against $s' = s$ and $F'$ which equals $F$ restricted to $A'_{n - i}$ and $B'_{n - i}$. 
        
        The same remark as for the ``commit to the root" rule applies here, as well. Namely we can assume $F'$ is properly defined and need to relabel $A'_{n - i}$ by elements of $P_{n - i} = \{0, ..., n - i\}$ and $B'_{n - i}$ by elements of $H_{n - i} = \{0, ..., n - i - 1\}$.
    \end{itemize}
    We denote the ``commit to the root" rule with $h$ chosen as above as \CMT{h} and similarly denote the ``forbid holes" rule as \FH{h_{1}, ..., h_{i}}. 
    
    One can naturally generalize \CMT{h} to \CMT{h_{1}, ..., h_{i}} rule so as to  be able to fallback to $m$ even smaller than $n - 1$ as in \FH{h_{1}, ..., h_{i}} case. However, during the proof, we will only need the case with $i = 1$. Thus, we omit the formal definition of \CMT{h_{1}, ..., h_{i}} for $i > 1$.

    We first demonstrate the usage of \FH{h_{1}, ..., h_{i}} rule to get rid of \textit{loops}. 
    
    We say that a pair $(p, h)$ is a \deff{loop} for $F$ iff $F(p, h) = p$.

    \begin{lemma}
        Assume a loop $(p, h)$ for $F$ exists. Then, there is a winning anti-strategy for \D against $(s, F)$ with $s \geq \max\{s_{0}, 2(n - 2) + 2\}$.
    \end{lemma}
    \begin{proof}
        Let \D apply \FH{h} and play without ever giving two contradictory answers in a row. If she is guaranteed not to encounter $p$, then there is a winning anti-strategy, as was shown above, against $(s, F)$ with $s \geq s_{0}$. 
        
        Otherwise, by answering \PP's questions without ever using $h$, she will eventually be asked $p$ for the first time. Since \D has not yet used $h$ and has never encountered $p$, the pair $(p, h)$ is consistent with all the answers given by \D so far, and so \D answers $h$. The record produced is $(p, h)$, and so \P will again ask $p$, since $F(p, h) = p$. \D answers $h$ and continues the loop until the number of rounds reaches $s$ and she wins.

        It remains to show that in the above case, \D can win against $s \geq 2(n - 2) + 2$. Notice that \D wins against $s > s'$, where $s'$ is the length of the shortest locally consistent path in $\mathcal{G}_{F}$ with the last node of the path being $p$, $p$ never appearing among the nodes except the last one and labels of all the edges of the path being different from $h$. We will show that $s' \leq 2(n - 2) + 1$.

        Let $p_{0}, p_{1}, ..., p_{s'}$ be nodes of the path as above and $h_{1}, ..., h_{s'}$ be labels of the corresponding edges, i.e. $h_{i}$ labels the edge going from $p_{i - 1}$ to $p_{i}$ for $i \leq s'$. Thus, $p_{0}$ is the initial node of $\mathcal{G}_{F}$ (i.e. $F(\emptyset) = p_{0}$) and $p = p_{s'}$. Notice that for $i < s'$, it holds that $p_{i} \neq p$. Since we are considering the shortest such path, it follows that for $p_{i} \neq p_{0}$ for all $i > 0$ (assuming the opposite, i.e. that $p_{i} = p_{0}$ for $i >0$, one can get a strictly shorter path $p_{i}, ..., p_{s'}$ satisfying all the above assumptions). 
        
        To finish the proof of $s' \leq 2(n - 2) + 1$, it is enough to show that each $p' \in P_{n} \setminus \{p_{0}, p\}$ appears at most twice in the given path.

        For the sake of contradiction, let $p_{i_{1}} = p_{i_{2}} = p_{i_{3}} \in P_{n} \setminus \{p_{0}, p\}$ with $0 < i_{1} < i_{2} < i_{3} < s'$. Let $P_{1}$ denote $p_{0}, p_{1}, ..., p_{i_{1}}$, $P_{2}$ denote $p_{{i_{1}}}, p_{i_{i} + 1}, ..., p_{i_{2}}$, $P_{3}$ denote $p_{{i_{2}}}, p_{i_{2} + 1}, ..., p_{i_{3}}$ and $P_{4}$ denote $p_{{i_{3}}}, p_{i_{3} + 1}, ..., p_{s'}$. So the path we have can be written as $P_{1} \circ P_{2} \circ P_{3} \circ P_{4}$ with $\circ$ denoting the usual composition of directed walks in a graph.

        Consider $h_{i_{1} + 1}, h_{i_{2} + 1}$ and $h_{i_{3} + 1}$ - labels of the edges connecting $p_{i_{j}}$ to $p_{i_{j} + 1}$ in the given path. We now show that all such labels must be distinct. 
        
        Assume, for example, that $h_{i_{1} + 1} = h_{i_{2} + 1}$. Then the path $P_{1} \circ P_{3} \circ P_{4}$ is locally consistent in $\mathcal{G}_{F}$ with the last node of the path being $p$, $p$ never appearing among the nodes except the last one and labels of all the edges of the path being different from $h$. Note that $P_{1} \circ P_{3} \circ P_{4}$ is strictly shorter than $P_{1} \circ P_{2} \circ P_{3} \circ P_{4}$ contradicting the fact that the latter path is the shortest among the ones satisfying the above condition. 
        
        Cases $h_{i_{1} + 1} = h_{i_{3} + 1}$ and $h_{i_{2} + 1} = h_{i_{3} + 1}$ are done analogously.

        Thus, $h_{i_{1} + 1}, h_{i_{2} + 1}$ and $h_{i_{3} + 1}$ are all distinct. In particular, either $h_{i_{2} + 1}$ or $h_{i_{3} + 1}$ are different from $h_{i_{1}}$. This means either $P_{1} \circ P_{3} \circ P_{4}$ or $P_{1} \circ P_{4}$ is a locally consistent path in $\mathcal{G}_{F}$ with last node of the path being $p$, $p$ never appearing among the nodes except the last one and labels of all the edges of the path being different from $h$. Both such paths are strictly shorter than $P_{1} \circ P_{2} \circ P_{3} \circ P_{4}$ and so we get a contradiction.
    \end{proof}

    Thus, we can assume there are no loops for $F$. 
    
    We begin the proof of Theorem \ref{main1} by generating a symmetric php-tree $\mathcal{T}_{F}$ as in \ref{php-tree canonical anti-strategy} and we assume $\mathcal{T}_{F}$ is complete. Notice that $\mathcal{T}_{F}$ does not necessarily capture all the information about $F$; namely, there may be two different functional parts of \PP's strategies $F, F'$ so that $\mathcal{T}_{F} = \mathcal{T}_{F'}$. 
    
    The above discussion leads to a notion of a \textit{loose pair} for $\mathcal{T}_{F}$. Let $p \in P_{n}$ and $h \in H_{n}$. 
    
    We say that $(p, h)$ is a \deff{loose pair} for $\mathcal{T}_{F}$ iff there is no node $v$ together with an edge $\{v, w\}$ connecting $v$ to its child $w$ so that $l_{V}(v) = p$ and $l_{E}(\{v, w\}) = h$.

    \begin{lemma}
        Assume there are no loose pairs for $\mathcal{T}_{F}$. Then, there is a winning anti-strategy for \D against $(s, F)$ for any $s \geq s_{0} + 1$.
    \end{lemma}
    \begin{proof}
        Recall that $\mathcal{T}_{F}$ we are considering is complete, and its root is labeled by $n$. Assuming no pair $(p, h)$ is loose, we derive that, for any possible record $(p, h)$, it holds that $F(p, h) \neq n$.
        Indeed, the fact that $(p, h)$ is not loose guarantees that there exists a node $v$ such that $l_V(v) = F(p,h)$. Since $\mathcal{T}_{F}$ is a php-tree, it follows that $v$ is not a root, and so cannot be labeled by $n$.

        Thus, \D may pick any possible $h \in H_{n}$ as an answer to $n$ and apply \CMT{h}. Since she is guaranteed not to encounter question $n$ anymore, we can construct a winning anti-strategy for \D by modifying one for $n - 1$. As we have already seen, this allows \D to win against any $s \geq s_{0} + 1$.
    \end{proof}

    Hence, we assume there is a loose pair $(p, h)$ for $\mathcal{T}_{F}$. From this point onward, the proof proceeds similarly to the one of \ref{main}, where we use derivations to construct covers-by-two. The difference is that, at this point, we have some additional structural information that we can leverage, and so the case analysis is much shorter. 
    
    We leave these derivations to Appendix \ref{php-tree derivations}.
    
\end{proof}

We finish this section by stating another open question.

\begin{defn}
    Let $n \geq 3$ and $P_{n} = \{0, ..., n\}$, $H_{n} = \{0, ..., n-1\}$. Let $(s, F)$ be \PP's strategy and $\alpha$ be a partial one-to-one mapping from $P_{n}$ to $H_{n}$. 
    
    We say that \DD's anti-strategy against $(s, F)$ is \deff{compatible with} $\alpha$ iff for any question $p$ and any answer $h$ given according to \DD's anti-strategy the set $\alpha \cup \{(p, h)\}$ is a partial one-to-one mapping. 
\end{defn}

\begin{cor}
    Let $n \geq 3$ and $(s, F)$ be \PP's strategy for $n$. Then, there is a partial one-to-one mapping $\alpha$ from $P_{n}$ to $H_{n}$ of size $n - 3$ for which there is a winning-anti-strategy for \D against $(s, F)$ which is compatible with $\alpha$. Moreover, $\alpha$ does not depend on $s$.
\end{cor}

\begin{question}
    Let $n \geq 3$ and $(s, F)$ as above. Let $\alpha$ be an arbitrary one-to-one mapping from $P_{n}$ to $H_{n}$ of size $n - 3$. Is there a winning anti-strategy for \D against $(s, F)$ which is compatible with $\alpha$?
\end{question}

\section{Acknowledgement}
\label{Acknowledgement}

The authors deeply thank Jan Kraj\'{i}\v{c}ek for his suggestions, which greatly improved the presentation of this article.
We also would like to thank Ond\v{r}ej Je\v{z}il for his comments on \S \ref{Analysis of simplified G2}, which sharpened our understandings of our toycase.
We are also grateful for Neil Thapen for his comments on $\mathcal{G}_{2}$ and suggestions on further readings.

The first author (Eitetsu Ken) was supported by JSPS KAKENHI Grant Number 22KJ1121, Grant-in-Aid for JSPS Fellows, and FoPM program at the University of Tokyo.

The second author (Mykyta Narusevych) was supported by the Charles University project PRIMUS/21/SCI/014, Charles University Research Centre program No. UNCE/24/SCI/022 and GA UK project No. 246223


\section{Appendix}\label{Appendix}
\subsection{Our formulation of $T^{2}_{2}(R)$-proofs and its propositional translation}\label{firstorderLK}

Let $\widetilde{\mathcal{L}}_{BA}$ be the language of bounded arithmetics adopted in \cite{Buss} to which we add:
\begin{itemize}
 \item the unary function symbol $len(s)$,
 \item the binary function symbols $(s)_{i}$ and $s*x$.
\end{itemize}
Their interpretations on the standard model is defined as follows.

 We encode a finite sequence $\sigma=(a_{1},\ldots, a_{l})$ of numbers as follows; 
 Let $i \in [l]$. Consider the binary representation $b_{1}^{(i)}\cdots b_{k_{i}}^{(i)}$ of $a_{i}$.
 Let $\widetilde{a}_{i}$ be the number having the following binary representation;
 \[\widetilde{a}_{i} = 1b_{1}^{(i)}\cdots 1b_{k_{i}}^{(i)}.\]
 We encode $\sigma$ by the number $\widetilde{\sigma}$, which is determined by the following binary representation;
 \[\widetilde{\sigma} = \widetilde{a}_{1} 00 \widetilde{a}_{2} \cdots 00\widetilde{a}_{l}.\]
 
 It is clear that $\sigma \mapsto \widetilde{\sigma}$ is injective.
 We adopt the convention that numbers out of the range of $\widetilde{\cdot}$ code the empty sequence.
 
 Based on this, we define:
 \begin{itemize}
  \item If $s = \widetilde{\sigma}$ for some $\sigma=(a_{1},\ldots, a_{l})$, then $len(s) := l$, $(s)_{i}:=a_{i}$ for $i \in [len(s)]$ and $(s)_{i} := 0$ for $i \not \in [len(s)]$. $s*x := \widetilde{(a_{1},\ldots,a_{l},x)}$.
  \item Otherwise, $len(s) = 0$, $(s)_{i}:=0$, $s*x := \widetilde{(x)}$.
 \end{itemize}

We consider the theory $\widetilde{T^{2}_{2}}(R)$ of the language $\widetilde{\mathcal{L}}_{BA}\cup\{R\}$, where $R$ is a fresh binary predicate symbol; whose axioms are listed below:
\begin{enumerate}
 \item BASIC given in \cite{Buss}.
 \item $len(0)=0$.
 \item $len(s * x) = len(s) + 1$.
 \item $len(s) \leq |s|$.
 \item $(s)_{i} \leq s$.
 \item $(s*x)_{len(s)+1} = x$.
 \item $1\leq i \land i \leq len(s) \rightarrow (s*x)_{i} = (s)_{i}$.
  \item $s \leq s*x$.
 \item $x \leq s*x$.
 \item $|s*x| \leq |s| + 2(|x|+1)+2$.
 \item $\widetilde{\Sigma^{b}_{2}}(R)$-Induction, where $\widetilde{\Sigma^{b}_{i}}(R)$ and $\widetilde{\Pi^{b}_{i}}(R)$ ($i \in \NN$) are classes of bounded formulae defined as follows:
 \begin{itemize}
  \item $\widetilde{\Sigma^{b}_{0}}(R)$ and $\widetilde{\Pi^{b}_{0}}(R)$ are the class of all sharply bounded formulae (of the extended language $\widetilde{\mathcal{L}}_{BA}\cup\{R\}$).
  \item $\widetilde{\Sigma^{b}_{i+1}}(R)$ is the class collecting bounded formulae of the following form;
  \[\exists x \leq t. \varphi(x),\]
  where $\varphi \in \widetilde{\Pi^{b}_{i}}(R)$.
  Similarly for $\widetilde{\Pi^{b}_{i+1}}(R)$.
 \end{itemize}
\end{enumerate}

It is routine to verify the following:
\begin{prop}
The following hold:
\begin{enumerate}
 \item For a formula $\varphi(x,y)$,
 \[\widetilde{T^{2}_{2}}(R) \vdash \left(\forall x \leq t. \forall y \leq u. \ \varphi(x,y)\right) \leftrightarrow \forall p \leq (0*t)*u.\ \varphi((p)_{1},(p)_{2}).\]
 \item For any formula of the form $\forall x \leq |t|. \exists y \leq u. \varphi(x,y)$, where $\varphi(x,y) \in \bigcup_{i=0}^{1}\widetilde{\Pi^{b}_{i}}(R)$, the following holds;
 \[\widetilde{T^{2}_{2}}(R) \vdash \forall x \leq |t|. \exists y \leq u. \varphi(x,y) \leftrightarrow \exists s \leq t\#16u^{2}.\  \varphi(x,(s)_{x}).\]
 \item For any $\Sigma^{b}_{2}(R)$-formula $\varphi$ in the sense of \cite{Buss}, there exists $\widetilde{\Sigma^{b}_{2}}(R)$-formula $\widetilde{\varphi}$ such that
 \[\widetilde{T^{2}_{2}}(R) \vdash \varphi \leftrightarrow \widetilde{\varphi}.\]
\end{enumerate}
\end{prop}

Hence, $\widetilde{T^{2}_{2}}(R)$ is an extension of $T^{2}_{2}(R)$ of usual formalization. 
Furthermore, it is also straightforward to see that it is actually a conservative extension of $T^{2}_{2}(R)$. (The all function symbols we added are representing polynomial time functions, and their basic properties such as the additional axioms above can be proven in $S^{1}_{2}$ under appropriate formalizations of them. See \cite{Buss}.)

Now, one-sided sequent-calculus formulation of $\widetilde{T^{2}_{2}}(R)$ can be considered.
We adopt the conventions and formalization given in \cite{OrdinalAnalysis}, that is, we only treat formulae of negation normal form and adopt the following derivation rules:

 \begin{itemize}
 \item Initial Sequent;
 
 \begin{prooftree}
 \AxiomC{}  \RightLabel{\quad (where $L$ is a literal)}
  \UnaryInfC{$\Gamma, L, \overline{L}$}

\end{prooftree}

 \item $\lor$-Rule; 
   \begin{prooftree}
 \AxiomC{$\Gamma, \varphi_{i_{0}}$} \RightLabel{\quad (where $\varphi_{1} \lor \varphi_{2} \in \Gamma$, $i_{0}=1,2$)}
  \UnaryInfC{$\Gamma$}
 \end{prooftree}

   \item $\exists$-Rule: 
   \begin{prooftree}
 \AxiomC{$\Gamma, \varphi(u)$}
  \RightLabel{\quad (where $\exists x. \varphi(x) \in \Gamma$)}
  \UnaryInfC{$\Gamma$}
 \end{prooftree}

 \item $\exists^{\leq}$-Rule; 
   \begin{prooftree}
 \AxiomC{$\Gamma, \varphi(u)$}
 \AxiomC{$\Gamma, u \leq t$}
  \RightLabel{\quad (where $\exists x \leq t. \varphi(x) \in \Gamma$, and $u$ is a term)}
  \BinaryInfC{$\Gamma$}
 \end{prooftree}
 
 \item $\land$-Rule;
    \begin{prooftree}
 \AxiomC{$\Gamma, \varphi_{1}$}
 \AxiomC{$\Gamma, \varphi_{2}$}
   \RightLabel{\quad (where $\varphi_{1} \land \varphi_{2} \in \Gamma$)}
  \BinaryInfC{$\Gamma$}
 \end{prooftree}

 \item $\forall$-Rule:
    \begin{prooftree}
 \AxiomC{$\Gamma, \varphi(a)$}
   \RightLabel{\quad (where $\forall x. \varphi(x) \in \Gamma$, and $a$ is an eigenvariable)}
  \UnaryInfC{$\Gamma$}
 \end{prooftree}

 \item $\forall^{\leq}$-Rule;
    \begin{prooftree}
 \AxiomC{$\Gamma, \overline{a \leq t}, \varphi(a)$}
   \RightLabel{\quad (where $\forall x \leq t. \varphi(x) \in \Gamma$, and $a$ is an eigenvariable)}
  \UnaryInfC{$\Gamma$}
 \end{prooftree}

 \item Axiom of $\widetilde{T^{2}_{2}}(R)$;
  \begin{prooftree}
 \AxiomC{$\Gamma, \overline{\varphi}$} \RightLabel{\quad (where $\varphi$ is a substitution instance of one of open axioms of $\widetilde{T^{2}_{2}}(R)$)}
  \UnaryInfC{$\Gamma$}
 \end{prooftree}

 \item $\widetilde{\Sigma^{b}_{2}}(R)$-Indunction;
     \begin{prooftree}
 \AxiomC{$\Gamma, \varphi(0)$}
 \AxiomC{$\Gamma, \overline{\varphi(a)},\varphi(a+1)$}
  \AxiomC{$\Gamma, \overline{\varphi(t)}$}
  \RightLabel{\quad (where $\varphi \in \widetilde{\Sigma^{b}_{2}}(R)$, $a$ is an eigenvariable)}
  \TrinaryInfC{$\Gamma$}
 \end{prooftree}
\end{itemize}

Furthermore, we consider the following first-order version of $ontoPHP^{n+1}_{n}$:

\begin{defn}
Let $ontoPHP^{x+1}_{x}(R)$ be the following sequent of $\widetilde{\mathcal{L}}_{BA}$-formulae:
\begin{align*}
\{ &\exists i \leq x+1. \forall j \leq x.\ \lnot R(i,j),\\ &\exists i \leq x+1. \exists i' \leq x+1. \exists j \leq x.  (R(i,j) \land R(i',j) \land i \neq i^{\prime})),\\
 &\exists j \leq x. \forall i \leq x+1.\ \lnot R(i,j),\\ &\exists j \leq x. \exists j' \leq x. \exists i \leq x+1.  (R(i,j) \land R(i',j) \land j \neq j^{\prime})) \}.
\end{align*}
\end{defn}

Now, Proposition \ref{ParisWilkieTranslation} can be proved by a straightforward application of the Paris-Wilkie translation argument. (See \cite{clearwitnessing} for a presentation under a similar convention to ours. For the theoretical background and context, see section 8.2 of \cite{proofcomplexity}.)

\subsection{An equivalent formulation of $\mathcal{G}_{2}$}\label{Auxiliary info not needed}
In this section, we see that the auxiliary information $\vec{A}$ in the formulation of $\mathcal{G}_{2}$ is not essential.

\begin{defn}
Let $n,C \in \NN$ and $T$ be an $(n,C)$-tree.
$\mathcal{G}^{\prime}_{2}(n,C,T)$ is the following game: the same as $\mathcal{G}_{2}(n,C,T)$ except:
\begin{itemize}
 \item A possible position is a map $L$ such that:
 \begin{enumerate}
 \item $\dom(L) \subseteq T$.
   \item For $v \in \dom(L)$, $L(v) \in (\mathcal{M}_{n})_{\leq |n|^{C}\times height(v)}$.
 
   \item $\dom(L)$ is closed downwards under $\subseteq$.
   \item $v \subseteq w \in \dom(L)$, if $M=L(v)$ and $M'=L(w)$, then they satisfy $M \subseteq M'$.

 \end{enumerate}
 \item $L_{0}:=\emptyset$.
 \item When \deff{Prover} plays $\langle o,x,B \rangle$, he plays a pair $\langle o,x \rangle$ instead.
 The transitions are made in the same way as $\mathcal{G}_{2}$, just disregarding the auxiliary information $\vec{A}$, $\vec{\alpha}$, and $B$.
\end{itemize} 
\end{defn}

As $\mathcal{G}_{2}$, the game $\mathcal{G}^{\prime}_{2}$ is also determined.
The proof is in exactly the same lines as Corollary \ref{determinacy} and we omit it.
The notions of \deff{Prover}'s oblivious strategies, \deff{Delayer}'s strategies, and which beats which are defined analogously with Definitions \ref{obliviousstrategy}-\ref{strategybeatsanother}, disregarding the auxiliary information $\vec{A}$.
Then $\mathcal{G}_{2}$ and $\mathcal{G}^{\prime}_{2}$ are equivalent in the following sense:

\begin{prop}
The following are equivalent:
\begin{enumerate}
 \item\label{withinfo} There exists $C>0$ such that, for sufficiently large $n$, there exists an $(n,C)$-tree $T$ such that \deff{Prover} has an oblivious winning strategy for $\mathcal{G}_{2}(n,C,T)$.
  \item\label{withoutinfo} The same condition for $\mathcal{G}'_{2}(n,C,T)$.
\end{enumerate}
\end{prop}

\begin{proof}
(\ref{withoutinfo}) $\Leftarrow$ (\ref{withinfo}) is easier.
Let $C$ and $n$ be as claimed in (\ref{withoutinfo}). 
Then there exists $T$ such that \deff{Prover} has an oblivious winning strategy $(f'_{1},f'_{2})$ for $\mathcal{G}'_{2}(n,C,T)$.
We observe that it induces \deff{Prover}'s oblivious winning strategy for $\mathcal{G}_{2}(n,C,T)$.
Indeed, set:
\begin{align*}
 f_{1}(v,M,\vec{A}) &:= f'_{1}(v,M),\\
 f_{2}(v,M,\vec{A},M') &:= \langle o, x, 1 \rangle, \ \mbox{where $\langle o,x\rangle := f'_{2}(v,M,M')$.} 
\end{align*}
Then $(f_{1},f_{2})$ beats arbitrary \deff{Delayer}'s strategy $g$ for $\mathcal{G}_{2}(n,C,T)$ since the play depends only on $g\restriction_{D}$, where
\[D:= \{(L,Q) \mid \mbox{For each $v \in \dom(L)$, $L(v)$ is of the form $\langle M, 1, \ldots, 1\rangle$}\},\]
and $(f'_{1},f'_{2})$ beats \deff{Delayer}'s strategy $g'$ for $\mathcal{G}'_{2}$ induced by $g \restriction_{D}$:
\[g'(L,Q) := g(\tilde{L},Q),\]
where $\dom(\tilde{L})=\dom(L)$ and $\tilde{L}(v) := \langle M, \vec{1} \rangle$ with $|\vec{1}|=height(v)$.

Now, we consider the converse.
Let $C$ be as claimed in (\ref{withinfo}). 
There exists $N$ such that, for all $n \geq N$, there exists an $(n,C)$-tree $T$ such that \deff{Prover} has an oblivious winning strategy for $\mathcal{G}_{2}(n,C,T)$.
Let $C' \geq C$ be a constant and $N' \geq N$ be large enough such that
\[\forall n \geq N'.\ 2^{|n|^{C}}(2^{|n|^{C}}+1)+2^{|n|^{C}} \leq 2^{|n|^{C'}}.\]
We show that, for all $n \geq N'$, there exists an $(n,C')$-tree $T'$ such that \deff{Prover} has an oblivious winning strategy $(f'_{1},f'_{2})$ for $\mathcal{G}'_{2}(n,C',T')$.
Take $T$ for $n$ assured by (\ref{withinfo}).
\deff{Prover} has an oblivious winning strategy $(f_{1},f_{2})$ for it.
To define $T'$, let
\[\llbracket k, A \rrbracket := A (2^{|n|^{C}}+1) + k \in [2^{|n|^{C'}}]\]
for $k,A \in [2^{|n|^{C}}]$.
Note that $\llbracket \cdot, \cdot \rrbracket$ is injective on $[2^{|n|^{C}}] \times [2^{|n|^{C}}]$.

Now, define $T'$ as follows:
\[T':=\left\{(\llbracket v_{1}, A_{1} \rrbracket,\ldots,\llbracket v_{h}, A_{h} \rrbracket) \in  [2^{|n|^{C'}}]^{\leq C} \mid (v_{1},\ldots, v_{h}) \in T,\ \& \forall j \in [h].\ A_{j} \in [2^{|n|^{C}}]\right\}.\]
For each $v'=(\llbracket v_{1}, A_{1} \rrbracket,\ldots,\llbracket v_{h}, A_{h} \rrbracket) \in T'$, let 
\[\vec{A}(v') := (A_{1}, \ldots, A_{h}).\]
Then the following $(f'_{1},f'_{2})$ is the desired oblivious winning strategy of \deff{Prover} for $\mathcal{G}'_{2}(n,C',T')$.

\begin{align*}
f'_{1}(v',M) &:= f_{1}(v,M,\vec{A}(v'))\\
f'_{2}(v',M,M') &:= 
\begin{cases}
\langle o, \llbracket x,B \rrbracket \rangle &\mbox{(if $o=1$)}\\
\left\langle o, \left(\llbracket x_{1}, A_{1}\rrbracket, \ldots, \llbracket x_{h}, A_{h}\rrbracket\right) \right\rangle \ &\mbox{(if $o$ is $2$ or $3$, $x=(x_{i})_{i=1}^{h} \in T$, and $\vec{A}(v')=(A_{i})_{i}$.)}
\end{cases}
\end{align*}
where $\langle o,x,B\rangle := f_{2}(v,M,\vec{A}(v'),M')$.

\end{proof}

\subsection{Derivation of labelings}
\label{php-tree derivations}

This subsection follows directly the proof of \ref{main1}. We are given \PP's strategy $(s, F)$ for $\mathcal{G}_{2}(n)$ with $s \geq \max \{s_{0} + 1, 2(n - 1) + 2\}$ and we know that there are winning anti-strategies for \D against any $(s', F')$ for $\mathcal{G}_{2}(n - 1)$ and $(s'', F'')$ for $\mathcal{G}_{2}(n - 2)$ with $s', s'' \geq s_{0}$. 

We have also argued that the associated php-tree $\mathcal{T}_{F}$ is complete, $F$ has no loops, and there is a loose pair $(p, h)$ for $\mathcal{T}_{F}$. 

Finally, the root of $\mathcal{T}_{F}$ is $n \in P_{n}$.

Without loss of generality let $(p, h)$ be $(0, 0) \in P_{n} \times H_{n}$. Let \D use \FH{0}. If $0 \in P_{n}$ is guaranteed not to appear as \PP's question, we construct an anti-strategy for \D against any $s \geq s_{0}$. So we can assume there is $p' \in P_{n}$ and $h' \in H_{n}$ with $p' \neq 0 \in P_{n}$ and $h' \neq 0 \in H_{n}$ such that $F(p', h') = 0$. 

Without loss of generality let $(p', h')=(1, 1)$.

Let \D use \FH{0, 1}. Similarly as above, we may assume there are $p'' \neq 0, 1$ and $h'' \neq 0, 1$ such that either $F(p'', h'') = 0$ or $F(p'', h'') = 1$. Here we use the fact that we have winning anti-strategies for \D for $\mathcal{G}_{2}(n - 2)$. Without loss of generality let $(p'', h'')=(2, 2)$.

We start by considering a subtree $\mathcal{T}'$ of $\mathcal{T}_{F}$ formed as follows: start at the root $n \in P_{n}$ and proceed without using edges labeled as $0, 1, 2 \in H_{n}$. Since $\mathcal{T}_{F}$ is complete, eventually, an $(n - 3)$-long branch of $\mathcal{T}_{F}$ is built, which can also be viewed as a globally consistent path of $\mathcal{G}_{F}$ of length $n - 3$. Denote this branch as $P$. Denote $P'$ and $P''$ branches of length $n - 2$ with $P$ as their initial segment and last edges labeled by $1$ or $2 \in H_{n}$. $\mathcal{T'}$ is then defined as a collection of all the maximal branches of $\mathcal{T}_{F}$ which have $P'$ or $P''$ as their initial segments. $\mathcal{T}'$ has a natural tree structure which is illustrated in Figure \ref{php2}.

\begin{figure}[h!t]
\centering
\begin{tikzpicture}[-, >=stealth, semithick, node distance=1.6cm, initial text="", initial/.style={draw=none}]
    \node[state] (qroot) {$n$};
    \node[state, initial, below of=qroot] (qnull) {...};
    \node[state, below of=qnull] (qinit) {};
    \node[state, below left of=qinit] (qL) {};
    \node[state, below right of=qinit] (qR) {};
    \node[state, below left of=qL] (qLL) {};
    \node [state, below of=qL] (qLR) {};
    \node[state, below of=qR] (qRL) {};
    \node [state, below right of=qR] (qRR) {};
    \node[state, below of=qLL] (qLLB) {};
    \node [state, below of=qLR] (qLRB) {};
    \node[state, below of=qRL] (qRLB) {};
    \node [state, below of=qRR] (qRRB) {};
    
    \draw   (qroot) edge[] node{} (qnull)
            (qnull) edge[] node{} (qinit)
            (qinit) edge[above left] node{1} (qL)
            (qinit) edge[above right] node{2} (qR)
            (qL) edge[above left] node{0} (qLL)
            (qL) edge[right] node{2} (qLR)
            (qR) edge[left] node{0} (qRL)
            (qR) edge[above right] node{1} (qRR)
            (qLL) edge[left] node{2} (qLLB)
            (qLR) edge[right] node{0} (qLRB)
            (qRL) edge[left] node{1} (qRLB)
            (qRR) edge[right] node{0} (qRRB);
\end{tikzpicture}
\caption{- representation of $\mathcal{T}'$}
\label{php2}
\end{figure}

Blank circles represent nodes for which we have not yet fixed any particular labeling. The proof proceeds by subsequently deriving more and more information in the form of labeling of nodes. Derivations proceed similarly as in Theorem \ref{main}.

Without loss of generality, we assume the top-most blank node as above is labeled as $3 \in P_{n}$ (notice that such a node cannot be labeled as $0, 1, 2 \in P_{n}$ as this would contradict the fact that $(0, 0)$ is loose for $\mathcal{T}_{F}$). 

Using the fact that $(0, 0)$ is loose, we also derive that nodes labeled by $0 \in P_{n}$ must appear on a branch of $\mathcal{T}'$ somewhere below an edge labeled as $0 \in H_{n}$. 

We further derive $F(3, 2) = 2$, otherwise $F(3, 2)$ equals $1$ and since $F(1, 1) = 0$ we get that $(0, 0)$ cannot be loose for $\mathcal{T}_{F}$. 

The result is shown in Figure \ref{php3}.

\begin{figure}[h!t]
\centering
\begin{tikzpicture}[-, >=stealth, semithick, node distance=1.6cm, initial text="", initial/.style={draw=none}]
    \node[state, initial] (qnull) {...};
    \node[state, below of=qnull] (qinit) {3};
    \node[state, below left of=qinit] (qL) {};
    \node[state, below right of=qinit] (qR) {2};
    \node[state, below left of=qL] (qLL) {};
    \node [state, below of=qL] (qLR) {};
    \node[state, below of=qR] (qRL) {};
    \node [state, below right of=qR] (qRR) {1};
    \node[state, below of=qLL] (qLLB) {};
    \node [state, below of=qLR] (qLRB) {0};
    \node[state, below of=qRL] (qRLB) {};
    \node [state, below of=qRR] (qRRB) {0};
    
    \draw
            (qnull) edge[] node{} (qinit)
            (qinit) edge[above left] node{1} (qL)
            (qinit) edge[above right] node{2} (qR)
            (qL) edge[above left] node{0} (qLL)
            (qL) edge[right] node{2} (qLR)
            (qR) edge[left] node{0} (qRL)
            (qR) edge[above right] node{1} (qRR)
            (qLL) edge[left] node{2} (qLLB)
            (qLR) edge[right] node{0} (qLRB)
            (qRL) edge[left] node{1} (qRLB)
            (qRR) edge[right] node{0} (qRRB);
\end{tikzpicture}
\caption{}
\label{php3}
\end{figure}

We consider two different cases.

\begin{enumerate}
        
    \item $F(2, 0) = 1$. We then label the rest of the blank nodes of $\mathcal{T}'$ using the fact that $\mathcal{T}_{F}$ is a complete symmetric php-tree. 
    
    Roman numerals in Figure \ref{php4} represent all the derivations and their order.

\begin{figure}[h!t]
\centering
\begin{tikzpicture}[-, >=stealth, semithick, node distance=2cm, initial text="", initial/.style={draw=none}]
    \node[state, initial] (qnull) {...};
    \node[state, below of=qnull] (qinit) {3};
    \node[state, below left of=qinit] (qL) {2 IV};
    \node[state, below right of=qinit] (qR) {2};
    \node[state, below left of=qL] (qLL) {1 V};
    \node [state, below of=qL] (qLR) {1 III};
    \node[state, below of=qR] (qRL) {1 I};
    \node [state, below right of=qR] (qRR) {1};
    \node[state, below of=qLL] (qLLB) {0 VI};
    \node [state, below of=qLR] (qLRB) {0};
    \node[state, below of=qRL] (qRLB) {0 II};
    \node [state, below of=qRR] (qRRB) {0};
    
    \draw
            (qnull) edge[] node{} (qinit)
            (qinit) edge[above left] node{1} (qL)
            (qinit) edge[above right] node{2} (qR)
            (qL) edge[above left] node{0} (qLL)
            (qL) edge[right] node{2} (qLR)
            (qR) edge[left] node{0} (qRL)
            (qR) edge[above right] node{1} (qRR)
            (qLL) edge[left] node{2} (qLLB)
            (qLR) edge[right] node{0} (qLRB)
            (qRL) edge[left] node{1} (qRLB)
            (qRR) edge[right] node{0} (qRRB);
\end{tikzpicture}
\caption{- sequence of derivations of labeling of $\mathcal{T}'$}
\label{php4}
\end{figure}

    Let $F(0, 1) = a \in P_{n}$. Based upon the value of $a$, we get different covers-by-two.

    \begin{itemize}
        \item $a \geq 3$ is shown in Figure \ref{php5}.

\begin{figure}[h!t]
\centering
\begin{tikzpicture}[->, >=stealth, semithick, node distance=2.4cm, initial text="", initial/.style={draw=none}]
    \node[state, initial] (qnull) {...};
    \node[state, right of=qnull] (q3) {3};
    \node[state, right of=q3] (q2) {2};
    \node[state, right of=q2] (q1) {1};
    \node[state, right of=q1] (q0) {0};
    \node[state, initial, below=2cm of qnull] (qnullP) {...};
    \node[state, right of=qnullP] (q3P) {3};
    \node[state, right of=q3P] (q2P) {2};
    \node[state, right of=q2P] (q1P) {1};
    \node[state, right of=q1P] (q0P) {0};

    \draw   (qnull) edge[above] node{} (q3)
            (q3) edge[above, color=red, text=black, dashed] node{2} (q2)
            (q2) edge[above] node{0} (q1)
            (q1) edge[above, color=red, text=black, dashed] node{2} (q0)
            (q0) edge[above, bend right=30] node{1} (qnull)
            (qnullP) edge[above] node{} (q3P)
            (q3P) edge[above] node{2} (q2P)
            (q2P) edge[above, color=red, text=black, dashed] node{1} (q1P)
            (q1P) edge[above] node{0} (q0P)
            (q0P) edge[above, bend right=30, color=red, text=black, dashed] node{1} (qnullP);
\end{tikzpicture}
\caption{- cover-by-two derived under the assumption $a \geq 3$}
\label{php5}
\end{figure}

        \item $a = 2$ is shown in Figure \ref{php6}.

\begin{figure}[h!t]
\centering
\begin{tikzpicture}[->, >=stealth, semithick, node distance=2.4cm, initial text="", initial/.style={draw=none}]
    \node[state, initial] (qnull) {...};
    \node[state, right of=qnull] (q3) {3};
    \node[state, right of=q3] (q2) {2};
    \node[state, right of=q2] (q1) {1};
    \node[state, right of=q1] (q0) {0};
    \node[state, initial, below of=qnull] (qnullP) {...};
    \node[state, right of=qnullP] (q3P) {3};
    \node[state, right of=q3P] (q2P) {2};
    \node[state, right of=q2P] (q1P) {1};
    \node[state, right of=q1P] (q0P) {0};

    \draw   (qnull) edge[above] node{} (q3)
            (q3) edge[above] node{1} (q2)
            (q2) edge[above] node{2} (q1)
            (q1) edge[above] node{0} (q0)
            (q0) edge[above, bend right=35, color=red, text=black, dashed] node{1} (q2)
            (qnullP) edge[above] node{} (q3P)
            (q3P) edge[above] node{2} (q2P)
            (q2P) edge[above] node{0} (q1P)
            (q1P) edge[above, color=red, text=black, dashed] node{2} (q0P)
            (q0P) edge[above, bend right=35] node{1} (q2P);
\end{tikzpicture}
\caption{- cover-by-two derived under the assumption $a = 2$}
\label{php6}
\end{figure}

        \item $a = 1$ is shown in Figure \ref{php7}.

\begin{figure}[h!t]
\centering
\begin{tikzpicture}[->, >=stealth, semithick, node distance=2.4cm, initial text="", initial/.style={draw=none}]
    \node[state, initial] (qnull) {...};
    \node[state, right of=qnull] (q3) {3};
    \node[state, right of=q3] (q2) {2};
    \node[state, right of=q2] (q1) {1};
    \node[state, right of=q1] (q0) {0};
    \node[state, initial, below of=qnull] (qnullP) {...};
    \node[state, right of=qnullP] (q3P) {3};
    \node[state, right of=q3P] (q2P) {2};
    \node[state, right of=q2P] (q1P) {1};
    \node[state, right of=q1P] (q0P) {0};

    \draw   (qnull) edge[above] node{} (q3)
            (q3) edge[above] node{1} (q2)
            (q2) edge[above] node{2} (q1)
            (q1) edge[above] node{0} (q0)
            (q0) edge[above, bend right=35, color=red, text=black, dashed] node{1} (q1)
            (qnullP) edge[above] node{} (q3P)
            (q3P) edge[above] node{2} (q2P)
            (q2P) edge[above] node{0} (q1P)
            (q1P) edge[above, color=red, text=black, dashed] node{2} (q0P)
            (q0P) edge[above, bend right=35] node{1} (q1P);
\end{tikzpicture}
\caption{- cover-by-two derived under the assumption $a = 1$}
\label{php7}
\end{figure}
        
    \end{itemize}

\clearpage

    \item $F(2, 0) = 0$. The resulting labeling of $\mathcal{T'}$ is illustrated in Figure \ref{php8}, where we additionally derive $F(0, 1) = 1$.

\begin{figure}[h!t]
\centering
\begin{tikzpicture}[-, >=stealth, semithick, node distance=1.5cm, initial text="", initial/.style={draw=none}]
    \node[state, initial] (qnull) {...};
    \node[state, below of=qnull] (qinit) {3};
    \node[state, below left of=qinit] (qL) {};
    \node[state, below right of=qinit] (qR) {2};
    \node[state, below left of=qL] (qLL) {};
    \node [state, below of=qL] (qLR) {};
    \node[state, below of=qR] (qRL) {0};
    \node [state, below right of=qR] (qRR) {1};
    \node[state, below of=qLL] (qLLB) {};
    \node [state, below of=qLR] (qLRB) {0};
    \node[state, below of=qRL] (qRLB) {1};
    \node [state, below of=qRR] (qRRB) {0};
    
    \draw
            (qnull) edge[] node{} (qinit)
            (qinit) edge[above left] node{1} (qL)
            (qinit) edge[above right] node{2} (qR)
            (qL) edge[above left] node{0} (qLL)
            (qL) edge[right] node{2} (qLR)
            (qR) edge[left] node{0} (qRL)
            (qR) edge[above right] node{1} (qRR)
            (qLL) edge[left] node{2} (qLLB)
            (qLR) edge[right] node{0} (qLRB)
            (qRL) edge[left] node{1} (qRLB)
            (qRR) edge[right] node{0} (qRRB);
\end{tikzpicture}
\caption{}
\label{php8}
\end{figure}

    At this point, the only possible values for $a \in P_{n}$, so that $F(3, 1) = a$, are $1$ or $2$.

    \begin{enumerate}
        \item $a = 1$. We arrive at the situation depicted in Figure \ref{php9}, where Roman numerals represent additional labelings and the order of their derivations as before.

\begin{figure}[h!t]
\centering
\begin{tikzpicture}[-, >=stealth, semithick, node distance=1.8cm, initial text="", initial/.style={draw=none}]
    \node[state, initial] (qnull) {...};
    \node[state, below of=qnull] (qinit) {3};
    \node[state, below left of=qinit] (qL) {1 I};
    \node[state, below right of=qinit] (qR) {2};
    \node[state, below left of=qL] (qLL) {0 III};
    \node [state, below of=qL] (qLR) {2 II};
    \node[state, below of=qR] (qRL) {0};
    \node [state, below right of=qR] (qRR) {1};
    \node[state, below of=qLL] (qLLB) {2 IV};
    \node [state, below of=qLR] (qLRB) {0};
    \node[state, below of=qRL] (qRLB) {1};
    \node [state, below of=qRR] (qRRB) {0};
    
    \draw
            (qnull) edge[] node{} (qinit)
            (qinit) edge[above left] node{1} (qL)
            (qinit) edge[above right] node{2} (qR)
            (qL) edge[above left] node{0} (qLL)
            (qL) edge[right] node{2} (qLR)
            (qR) edge[left] node{0} (qRL)
            (qR) edge[above right] node{1} (qRR)
            (qLL) edge[left] node{2} (qLLB)
            (qLR) edge[right] node{0} (qLRB)
            (qRL) edge[left] node{1} (qRLB)
            (qRR) edge[right] node{0} (qRRB);
\end{tikzpicture}
\caption{}
\label{php9}
\end{figure}

        This leads to a cover-by-two, as is shown in Figure \ref{php10}.

\begin{figure}[ht]
\centering
\begin{tikzpicture}[->, >=stealth, semithick, node distance=2.4cm, initial text="", initial/.style={draw=none}]
    \node[state, initial] (qnull) {...};
    \node[state, right of=qnull] (q3) {3};
    \node[state, right of=q3] (q2) {2};
    \node[state, right of=q2] (q1) {1};
    \node[state, right of=q1] (q0) {0};
    \node[state, initial, below=2cm of qnull] (qnullP) {...};
    \node[state, right of=qnullP] (q3P) {3};
    \node[state, right of=q3P] (q1P) {1};
    \node[state, right of=q1P] (q2P) {2};

    \draw   (qnull) edge[above] node{} (q3)
            (q3) edge[above] node{2} (q2)
            (q2) edge[above] node{1} (q1)
            (q1) edge[above] node{0} (q0)
            (q0) edge[above, bend right=35, color=red, text=black, dashed] node{1} (q1)
            (qnullP) edge[above] node{} (q3P)
            (q3P) edge[above] node{1} (q1P)
            (q1P) edge[above] node{2} (q2P)
            (q2P) edge[above, bend right=35, color=red, text=black, dashed] node{1} (q1P);
\end{tikzpicture}
\caption{- cover-by-two derived under the assumption $F(3, 1) = 1$}
\label{php10}
\end{figure}

        \item $a = 2$. Additional derivations and their order are illustrated in Figure \ref{php11}.

\begin{figure}[h!t]
\centering
\begin{tikzpicture}[-, >=stealth, semithick, node distance=2cm, initial text="", initial/.style={draw=none}]
    \node[state, initial] (qnull) {...};
    \node[state, below of=qnull] (qinit) {3};
    \node[state, below left of=qinit] (qL) {2 I};
    \node[state, below right of=qinit] (qR) {2};
    \node[state, below left of=qL] (qLL) {0 III};
    \node [state, below of=qL] (qLR) {1 II};
    \node[state, below of=qR] (qRL) {0};
    \node [state, below right of=qR] (qRR) {1};
    \node[state, below of=qLL] (qLLB) {1 IV};
    \node [state, below of=qLR] (qLRB) {0};
    \node[state, below of=qRL] (qRLB) {1};
    \node [state, below of=qRR] (qRRB) {0};
    
    \draw
            (qnull) edge[] node{} (qinit)
            (qinit) edge[above left] node{1} (qL)
            (qinit) edge[above right] node{2} (qR)
            (qL) edge[above left] node{0} (qLL)
            (qL) edge[right] node{2} (qLR)
            (qR) edge[left] node{0} (qRL)
            (qR) edge[above right] node{1} (qRR)
            (qLL) edge[left] node{2} (qLLB)
            (qLR) edge[right] node{0} (qLRB)
            (qRL) edge[left] node{1} (qRLB)
            (qRR) edge[right] node{0} (qRRB);
\end{tikzpicture}
\caption{}
\label{php11}
\end{figure}

        This leads to a cover-by-two, as is shown in Figure \ref{php12}.

\begin{figure}[h!t]
\centering
\begin{tikzpicture}[->, >=stealth, semithick, node distance=2.4cm, initial text="", initial/.style={draw=none}]
    \node[state, initial] (qnull) {...};
    \node[state, right of=qnull] (q3) {3};
    \node[state, right of=q3] (q2) {2};
    \node[state, right of=q2] (q0) {0};
    \node[state, right of=q0] (q1) {1};
    \node[state, initial, below of=qnull] (qnullP) {...};
    \node[state, right of=qnullP] (q3P) {3};
    \node[state, right of=q3P] (q2P) {2};
    \node[state, right of=q2P] (q0P) {0};
    \node[state, right of=q0P] (q1P) {1};

    \draw   (qnull) edge[above] node{} (q3)
            (q3) edge[above] node{2} (q2)
            (q2) edge[above] node{0} (q0)
            (q0) edge[above] node{1} (q1)
            (q1) edge[above, bend right=35, color=red, text=black, dashed] node{0} (q0)
            (qnullP) edge[above] node{} (q3P)
            (q3P) edge[above] node{2} (q2P)
            (q2P) edge[above] node{0} (q0P)
            (q0P) edge[above, color=red, text=black, dashed] node{2} (q1P)
            (q1P) edge[above, bend right=35] node{1} (q0P);
\end{tikzpicture}
\caption{- cover-by-two derived under the assumption $a = 2$}
\label{php12}
\end{figure}
        
    \end{enumerate}

\end{enumerate}
This finishes the case analysis. Notice that all the winning anti-strategies produced by covers-by-two as above allow \D to win against any possible $s$.

\end{document}